\documentclass[11pt]{amsart}
\usepackage{placeins}
\usepackage{amscd}
\usepackage{amsmath, amssymb, comment}
\usepackage{amsfonts}
\usepackage{enumerate}
\usepackage{color}
\usepackage{graphicx}
\usepackage{float}
\usepackage[font=small,labelfont=bf]{caption}
\usepackage{url}

\newcommand{\ov}[1]{\overline{#1}}

\newcommand{\ti}[1]{\widetilde{#1}}
\newcommand{\vp}{\varphi}

\newcommand{\e}{\varepsilon}

\newcommand{\p}{\partial}
\newcommand{\pbar}{\overline{\partial}}
\renewcommand{\e}{\varepsilon}

\newcommand{\abs}[1]{\left\lvert#1\right\rvert}

\renewcommand{\leq}{\leqslant}
\renewcommand{\geq}{\geqslant}

\newcommand{\be}{\begin{equation}}
\newcommand{\ee}{\end{equation}}

\newcommand{\R}{\mathbb{R}}
\newcommand{\C}{\mathbb{C}}
\newcommand{\MA}{\mathrm{MA}}

\begin{document}
\newcounter{remark}
\newcounter{theor}
\setcounter{theor}{1}
\newtheorem{claim}{Claim}
\newtheorem{theorem}{Theorem}[section]
\newtheorem{lemma}[theorem]{Lemma}
\newtheorem{corollary}[theorem]{Corollary}
\newtheorem{proposition}[theorem]{Proposition}
\newtheorem{question}{Question}[section]
\newtheorem{goal}{Goal}[section]
\newtheorem{definition}[theorem]{Definition}
\newtheorem{remark}[theorem]{Remark}

\numberwithin{equation}{section}


\setlength{\floatsep}{6pt plus 2pt minus 2pt}
\setlength{\textfloatsep}{6pt plus 2pt minus 2pt}
\setlength{\intextsep}{6pt plus 2pt minus 2pt}
\captionsetup[figure]{skip=3pt}

\renewcommand{\topfraction}{0.9}
\renewcommand{\bottomfraction}{0.9}
\renewcommand{\textfraction}{0.07}
\renewcommand{\floatpagefraction}{0.7}
\setcounter{topnumber}{4}
\setcounter{bottomnumber}{4}
\setcounter{totalnumber}{8}

\makeatletter
\setlength{\@fptop}{0pt}
\setlength{\@fpsep}{10pt plus 2pt minus 2pt}
\setlength{\@fpbot}{0pt plus 1fil}
\makeatother

\title[Asymptotics]{Asymptotics for the $k$-Hessian Eigenvalue on the Unit Ball}

\author[N. McCleerey]{Nicholas McCleerey}
\address{Department of Mathematics,
Purdue University, West Lafayette
150 N University St
West Lafayette, IN 47907}
\email{nmccleer@purdue.edu}

\author[A. Pande]{Abhinav Pande}
\address{Department of Mathematics,
Purdue University, West Lafayette
150 N University St
West Lafayette, IN 47907}
\email{pande1@purdue.edu}

\author[T. Wanasinghe]{Thidas Wanasinghe}
\address{Department of Mathematics,
Purdue University, West Lafayette
150 N University St
West Lafayette, IN 47907}
\email{cwanasin@purdue.edu}

\begin{abstract}
We compute the limit of the eigenvalue of the $k$-Hessian operator on the unit ball in $\R^n$ when $k\rightarrow \infty$, assuming that the ratio $\frac{n}{k}$ remains fixed. When $n< 2ke$, we moreover identify the limit of the corresponding eigenfunctions. We also derive monotonicity results and consider when the ratio varies in $n$. 
\end{abstract}


\maketitle


\section{Introduction}

\subsection{Main Results} Let $n \geq  k \geq 1$ be integers, and $\Omega\subset \R^n$ a bounded, $(k-1)$-convex domain. If $\Omega$ is sufficiently regular, then there is a unique $k$-convex function $u_{k,n}$ and unique constant $\lambda_{k,n} > 0$ which solve the {\bf first eigenvalue problem} of the $k$-Hessian operator $S_k$:
\begin{equation}\label{k eigenvalue}
\begin{cases}
S_k(u_{k,n}) = (-\lambda_{k,n} u_{k,n})^k &\text{ in }\Omega\\
u_{k,n} = 0 &\text{ on } \p\Omega\\
\inf_\Omega u_{k,n} = -1.&
\end{cases}
\end{equation}
We recall the definition of $S_k$, $k$-convexity, etc., in Section 2. Existence and uniqueness are due to Lions \cite{Lio85}, when $k=n$, and Wang \cite{W94} otherwise.

Recently, there has been great interest in studying the properties of the {\bf first $(k,n)$-eigenvalue}, $\lambda_{k,n}$, cf. \cite{LS17, Le18, BZ23a, BZ23b, Le25, LZ25, Le26, CLM26} to list a few recent works. An open question has been to determine estimates for $\lambda_{k,n}$; so far, little seems to be known, with even numerical methods limited to dimensions $2$ and $3$ \cite{GLLQ20, LLQ22, CLYY25}.\\

We will fix $\Omega = B_1^n\subset \R^n$ the unit ball from now on. In \cite[Thrm. 1.1 (i)]{Le25}, Le observed that the first $(n,n)$-eigenvalue satisfies:
\[
4\cdot (2n)^{-\frac{1}{n}} \leq \lambda_{n,n} \leq 4,\quad\quad \text{ for all }n\geq 1.
\]
In particular, since the left-hand side increases to $4$, Le showed that $\lambda_{n,n} \rightarrow 4$ as $n\rightarrow \infty$. Since the ball is known to maximize $\lambda_{n,n}$ over all convex domains with a fixed volume \cite{BNT09, Le18}, this provides a universal upperbound; by normalizing the domain with respect to an appropriate quermass, one also gets dimensional lower bounds for smooth, strictly convex domains \cite{Tso89}.

This asymptotic convergence of $\lambda_{n,n}$ is quite different from the well-known asymptotics for $\lambda_{1,n}$, which grows like $n^2$ as $n\rightarrow\infty$. In between these two extremes, the asymptotic behaviour of $\lambda_{k,n}$ is only partially known. In \cite[Lem. 5.2, Thrm. 7.1]{BP21} and \cite[Thrm. 1.1 (ii)]{Le25}, it is shown that the $\lambda_{k,n}$ remain bounded as $n\rightarrow\infty$ if and only if the ratios $\alpha := \frac{n}{k}$ also remain bounded. Le \cite[Cor. 1.3]{Le25} also shows that the $\lambda_{k,n}$ converge to $4$ if the ratios converge to 1. It is thus natural to ask if the $\lambda_{k,n}$ converge whenever the ratios do \cite[Rmk. 1.4]{Le25}.\\

Our first main result answers this in the affirmative, and computes, for the unit ball, the limit of the $\lambda_{k,n}$ as a function of the limiting ratio:
\begin{theorem}\label{limit theorem intro}
Suppose that $(n_i), (k_i)$, $1 \leq k_i \leq n_i$, are sequences such that $\lim_{i\rightarrow\infty} n_i = \infty$ and also:
\[
\lim_{i\rightarrow\infty} \frac{n_i}{k_i} = \alpha < \infty.
\]
Then
\[
\lim_{i\rightarrow\infty} \lambda_{k_i, n_i} = \frac{\alpha^{\alpha+1}}{(\alpha-1)^{\alpha-1}} \left(\frac{2}{\alpha}\right)^{\frac{2}{2-\alpha}}.
\]
\end{theorem}

When $\alpha = 1$ or $2$, the right-hand side should be understood as the limit of the above expression when $\alpha\rightarrow 1$ (resp. $2$), which is $4$ (resp. $8 e$). In particular, we recover \cite[Thrm. 1.1 (i) and Cor. 1.3]{Le25}.\\

We briefly discuss the proof before continuing. The starting point is to use the radial symmetry of the ball to reduce \eqref{k eigenvalue} to the following ODE:
\begin{equation}\label{radial k eigenvalue}
\begin{cases}
{n-1 \choose k-1} \left( \frac{u'_{k,n}}{r}\right)^{k-1}\left(u''_{k,n} + \frac{n-k}{k}\frac{u'_{k,n}}{r}\right) = (-\lambda_{k,n} u_{k,n})^k\\
u(0) = -1\quad\text{ and }\quad u(1) = 0.
\end{cases}
\end{equation}
This ODE makes sense for any $n \geq k \geq 1$, provided we extend the definition of the binomial coefficient ${n-1 \choose k-1}$ by using the Gamma function; the non-integer equation is still variational in nature, and in particular, still admits a characterization of $\lambda_{k,n}$ as a Rayleigh quotient. Verifying these claims is straightforward, but lengthy, and we have opted to omit proofs in order to keep the paper focused.

The Rayleigh quotient formula \cite{Tso90, W94} for $\lambda_{k,n}$ is the main tool used for showing Theorem \ref{limit theorem intro}. The lower bound can be deduced in a manner analogous to that in \cite{Le25}, with minor tweaks. The upper bound follows from using a new test function, which we now define. Set, for $\alpha \not= 2$:
\begin{equation}\label{r alpha and A alpha defn}
r_\alpha := \left(\frac{\alpha}{2}\right)^{\frac{1}{2-\alpha}}, \quad\quad\text{ and }\quad\quad A_\alpha := \frac{2}{2-\alpha}e^{-\alpha/2}.
\end{equation}
Then we define the piecewise function:
\begin{equation}\label{v alpha defn}
v_\alpha(r) := \begin{cases} - e^{-\frac{\alpha r^2}{2r_\alpha^2}} &\text{ for } 0 \leq r\leq r_\alpha \\ A_\alpha (r^{2-\alpha} - 1) &\text{ for }r_\alpha < r \leq 1. \end{cases}
\end{equation}
The case $\alpha = 2$ can be deduced from standard limits -- specifically, we set:
\[
r_\alpha := e^{-\frac{1}{2}},\quad\quad A_\alpha := 2e^{-1}, \quad\text{ and}
\]
\begin{equation}\label{v 2 defn}
v_\alpha(r) := \begin{cases} - e^{-e r^2} &\text{ for } 0 \leq r\leq e^{-1/2} \\  2 e^{-1} \log r &\text{ for } e^{-1/2} < r \leq 1. \end{cases} \tag{\ref{v alpha defn}$'$}
\end{equation}

In Section 3, we show that, on all of $(0,1)$, $v_\alpha$ is $C^2$, increasing, and $\alpha$-convex (which is the suitable ODE version of $k$-convexity for our context; see Definition \ref{alpha convex}). Since it also satisfies the requisite boundary conditions, its Rayleigh quotient naturally provides an upper bound for the $\lambda_{k_i,n_i}$, which we show to be asymptotically exact by using Laplace's method; a lower bound follows from a similar estimate to \cite{Le25}.\\

Although it may appear somewhat mysterious at first, $v_\alpha$ arises in a very natural way -- our second main result shows that, for small enough $\alpha$, it is the limit of the first eigenfunctions when the ratio $\alpha$ remains fixed:
\begin{theorem}\label{eigenfunction theorem intro}
Suppose that $1\leq \alpha < 2e$. Then $u_{k, \alpha k}$ converges to $v_\alpha$ in $C^{1,\beta}([0,1])$ as $k\rightarrow\infty$ for any $\beta\in(0,1)$.
\end{theorem}

This improves upon \cite[Prop. 1.5]{Le25}, where it was shown that, for large $n$, $u_{n,n}$ differs from the linear function $r\mapsto r - 1$ by a quantified amount, independent of $n$. We expect that Theorem \ref{eigenfunction theorem intro} holds for all $\alpha \geq 1$; actually, we expect $u_{k, \alpha k}$ to increase to $v_\alpha$ as $k\rightarrow\infty$ for any $\alpha\geq 1$. See Section 7 for further discussions.\\

The key to showing Theorem \ref{eigenfunction theorem intro} is to bound the quantity:
\[
w_{k,n} := \frac{\alpha u_{k,n}'}{r\left(u''_{k,n} + \frac{n-k}{k}\frac{u'_{k,n}}{r}\right)},
\]
which is (up to a constant) the ratio of the terms on the left-hand side of \eqref{radial k eigenvalue}. We show in Proposition \ref{ratio ODE} that $w_{k,n}$ satisfies a first order ODE which is nearly linear -- dropping the non-linear term yields a differential inequality which implies directly that
\begin{equation}\label{lower w bound intro}
1 \leq w_{k,n} \text{ on all of }[0, 1].
\end{equation}
This bound holds for all $\alpha \geq 1$ and $k\geq 1$, see Proposition \ref{lower w bound}.

It is more difficult to show an upperbound -- we show in Proposition \ref{upper w bound} that, for $\alpha < 2e$ and $k$ sufficiently large, it holds that:
\begin{equation}\label{upper w bound intro}
w_{k,n} \leq \frac{r_\alpha^2}{r_\alpha^2 - r^2}\text{ on }[0, r_\alpha).
\end{equation}
Notice that the right-hand side is what one gets from replacing $u_{k,n}$ with $v_\alpha$ in the definition of $w_{k,n}$.

From these bounds, one then shows Theorem \ref{eigenfunction theorem intro} in two parts -- first, on $[0, r_\alpha]$, the bounds imply that the $k^\text{th}$-root of the equation \eqref{radial k eigenvalue} tends to the first order ODE:
\[
\frac{u'_\infty}{r} = -\frac{\alpha}{r_\alpha^2} u_\infty,
\] 
whose solution is exactly $v_\alpha$. Next, we show that \eqref{upper w bound intro} implies an asymptotically correct lower bound for $u_{k, n}'$ on the same interval. Then, by combining these first two steps with the fact that $v_\alpha'' + (\alpha -1)\frac{v_\alpha'}{r} = 0$ on $(r_\alpha, 1]$, we can extend convergence to the full interval.\\

For our last main result, we show that \eqref{lower w bound intro} actually implies a monotonicity result for the $\lambda_{k, n}$, again when the ratio is fixed:
\begin{theorem}\label{monotone theorem intro}
Fix $\alpha \geq 1$ and $1\leq k < \ell$. Then:
\[
{\alpha k\choose k}^{-\frac{1}{k}}\, \lambda_{k,\alpha k}\, \leq\, {\alpha \ell \choose \ell}^{-\frac{1}{\ell}}\, \lambda_{\ell, \alpha \ell}.
\]
\end{theorem}
In particular, combining this with Theorem \ref{limit theorem intro} implies that the sequence  ${\alpha k\choose k}^{-\frac{1}{k}}\lambda_{k,\alpha k}$ increases to $\frac{\alpha}{r_\alpha^2}$. In Proposition \ref{binom prop}, it is shown that the function $k\mapsto {\alpha k \choose k}^{\frac{1}{k}}$ is strictly increasing for any $\alpha > 1$; as such, we immediately have:

\begin{corollary}\label{monotone cor intro}
For any $\alpha \geq 1$ and $1\leq k < \ell$, we have
\[
\lambda_{k,\alpha k} \leq \lambda_{\ell, \alpha\ell}.
\]
\end{corollary}

\begin{remark}\label{binom remark}
Using standard estimates for binomial coefficients and Proposition \ref{lower bound}, cf. \cite{Le25}, we deduce the concrete bounds for $\alpha > 2$:
\[
(\alpha-2) \alpha^{1-\frac{1}{k}} \leq \lambda_{k,\alpha k} \leq \frac{\alpha^{\alpha + 1}}{(\alpha-1)^{\alpha-1}}\left(\frac{2}{\alpha}\right)^{\frac{2}{2-\alpha}}.
\]
Similar bounds can be derived for $\alpha \leq 2$.\\
\end{remark}

We conclude with a few more remarks. First, our results adapt immediately to the first eigenvalue of the complex $k$-Hessian equation, since the radial reduction of this equation is of the same form as \eqref{radial k eigenvalue}, up to a constant; we discuss the specifics in Section 6.

Second, there still appear to be several follow-up questions one could consider. For instance, in line with Theorem \ref{monotone theorem intro}, numerical evidence suggests that the $u_{k, \alpha k}$ actually increase to $v_\alpha$, for all $\alpha \geq 1$. We include further discussions in Section 7, along with a few numerical approximations for various $\lambda_{k,n}$; for large $n$, these are new, even in the case when $k = n$.

%
%
%
%
%
%

\subsection{Outline}

In Section 2, we collect necessary background definitions and results. In Section 3, we prove Theorem \ref{limit theorem intro}; in Section 4, we prove Theorem \ref{monotone theorem intro}; and in Section 5, we prove Theorem \ref{eigenfunction theorem intro}. In Section 6, we indicate how our results apply to the complex setting. In Section 7, we discuss potential follow-up questions and include some numerical computations.

\medskip

{\bf Statement on AI Usage:}  Generative AI was used to prove Proposition \ref{binom prop} (ChatGPT 5.6), produce the figures in Section 7 (Claude {Opus 4.8)}, and proofread the paper (ChatGPT 5.6). It was not used for any writing, for production of the numerical solver, or to prove other results; the authors retain sole responsibility for the veracity of the arguments.

\medskip

{\bf Acknowledgments:}  This project was carried out as a part of the Purdue Experimental Math Lab in Spring 2026; we extend our sincere thanks to the organizer, Thomas Sinclair, and also to Sina Nadi, who served as the graduate student mentor. The first author would also like to thank Nam Le, Laszlo Lempert, and G\'abor Sz\'ekelyhidi for interesting discussions. 

\section{Background}

\subsection{Definitions and Notation}

Let $1\leq k\leq n$ be integers for the moment. We set $\Omega = B_1^n\subset \R^n$ to be the unit ball, centered at the origin. Given $u\in C^\infty(B_1^n)$, the {\bf $k$-Hessian operator}, acting on $u$, is:
\[
S_{k,n}(u) := \sigma_k(\lambda),
\]
where $\lambda = (\lambda_1, \ldots, \lambda_n)$ are the eigenvalues of the Hessian of $u$ and:
\[
\sigma_k(\lambda) = \sum_{i_1 < \ldots < i_k} \lambda_{i_1}\cdots \lambda_{i_k}
\]
is the $k^\text{th}$ elementary symmetric polynomial. Note that $S_{1, n}(u) = \Delta u = \sum_i \lambda_i$, is the Laplacian, and $S_{n, n}(u) = \MA(u) = \lambda_1\cdots \lambda_n$, is the Monge-Amp\`ere operator.

The operator $S_{k,n}$ is fully non-linear, and thus elliptic only on a certain cone of functions, known as the {\bf $k$-convex cone}. Specifically, $u\in C^\infty(B_1^n)$ is {\bf $k$-convex} if:
\[
S_{i,n}(u) \geq 0 \quad\text{ for all integers }1\leq i \leq k.
\]
It is well known that $u$ is $n$-convex iff it is convex in the usual sense.

Frequently, we will use various binomial coefficients; since we further work with non-integer parameters, these should be understood using the Gamma function. Specifically, recall that, for an integer $m\geq 0$, we have
\[
\Gamma(m + 1) := m!,
\]
and that $\Gamma(x)$ provides a smooth extension of the factorial to all $x > 0$. For any $0 \leq k \leq n$, not necessarily integer, we then set:
\[
{n \choose k} := \frac{\Gamma(n+1)}{\Gamma(k+1) \Gamma(n-k + 1)}.
\]
For the sake of brevity, we write $C_{k,n} := {n \choose k}$.\\

We will also need to use the Beta function, which we recall is defined for $x, y > 0$ as:
\[
B(x,y) := \int_0^1 s^{x-1} (1-s)^{y-1}\, ds.
\]
Additionally, $B$ can be expressed in terms of $\Gamma$ or the binomial coefficients as:
\begin{equation}\label{beta to binom}
B(x, y) = \frac{\Gamma(x)\Gamma(y)}{\Gamma(x + y)} = \frac{x+y}{xy} {x + y \choose y}^{-1}.
\end{equation}

\subsection{Radial $k$-convex functions}

We briefly derive equation \eqref{radial k eigenvalue}, along with its divergence form, from the definition of $S_{k,n}$.

\begin{proposition}\label{radial k Hessian proposition}
Suppose that $u$ is a smooth, radially symmetric function on $B_1^n\subset \R^n$, i.e. $u(x) = u(r)$, where $r = \sqrt{\sum_i x_i^2}$. Then for any integer $1\leq k\leq n$, we have:
\begin{align*}
S_{k,n}(u) &= {n-1 \choose k-1}\left(\frac{u'}{r}\right)^{k-1}\left(u'' + \frac{n-k}{k} \frac{u'}{r}\right)\\
&= \frac{C_{k,n}}{n}\, r^{1-n} \frac{d}{dr}\left( r^{n-k} (u')^{k}\right).
\end{align*}
\end{proposition}
\begin{proof}
If $u$ is radial, then its Hessian matrix $(u_{ij})$ is of the form:
\[
u_{ij} = \frac{u'}{r} \delta_{ij} + \left(u'' - \frac{u'}{r}\right)\frac{x_i x_j}{r^2},
\]
which has eigenvalues $u''$, with multiplicity $1$, and $\frac{u'}{r}$, with multiplicity $n-1$. It follows from the definition that:
\begin{align*}
S_{k,n}(u) &= {n-1 \choose k-1}\left(\frac{u'}{r}\right)^{k-1} u'' + {n-1\choose k} \left(\frac{u'}{r}\right)^{k}\\
& = {n-1 \choose k-1}\left(\frac{u'}{r}\right)^{k-1}\left(u'' + \frac{n-k}{k} \frac{u'}{r}\right).
\end{align*}
To see the divergence form, expand:
\[
\frac{r^{1-n}}{k} \frac{d}{dr}\left(r^{n-k}(u')^k\right) = \frac{n-k}{k} \left(\frac{u'}{r}\right)^k + \left(\frac{u'}{r}\right)^{k-1} u''
\]
and use that $\frac{1}{k}C_{k-1, n-1} = \frac{1}{n}C_{k,n}$. \qedhere \\
\end{proof}

\subsection{Eigenvalues for the radial $k$-Hessian}

As mentioned in the introduction, we will extend the radial $k$-Hessian operator to arbitrary parameters $n$ and $k$, using the formula in Proposition \ref{radial k Hessian proposition}. The appropriate notion of $k$-convexity for these non-integer operators is clear:
\begin{definition}\label{alpha convex}
Suppose that $1 \leq k \leq n$; set $\alpha := \frac{n}{k}$. We say $u \in C^2(0,1)$ is {\bf $\alpha$-convex} if $\limsup_{r\rightarrow 0} \frac{u'}{r} < \infty$, and both:
\[
u'(r) \geq 0 \quad\text{ and }\quad u'' + (\alpha - 1)\frac{u'}{r} \geq 0. 
\]
\end{definition}
\noindent Clearly $S_{k,n}(u) \geq 0$ for $\alpha$-convex $u$.

We now extend the eigenvalue problem to general $1 \leq k \leq n$:
\begin{definition}
We say that an $\alpha$-convex function $u$ and a real number $\lambda > 0$ solve the {\bf radial $(k,n)$-eigenvalue problem} if they satisfy \eqref{radial k eigenvalue}.
\end{definition}

To see the variational structure of $S_{k,n}$, we define the {\bf $\alpha$-energy} to be:
\[
E_{k,n}[u] := \frac{1}{k+1} \int_0^1 (-u)\, S_{k,n}(u)\ r^{n-1}dr. 
\]
Note the factor of $r^{n-1}$ in the integral.

Standard computations show that, if $u_t = u + tv$, for $t \in (-\e, \e)$, $u, v$ smooth, and $v$ compactly supported, then:
\[
\frac{d}{dt}\Big|_{t=0} E_{k,n}[u_t] = \int_0^1 (-v) S_{k,n}(u)\ r^{n-1}dr.
\]
For $\alpha$-convex $u$ with $u(1) = 0$, it will often be convenient to use the alternate expression:
\begin{equation}\label{C expression for E}
E_{k,n}[u] = \frac{C_{k,n}}{n(k+1)} \int_0^1 r^{n-k}(u')^{k+1}\, dr,
\end{equation}
which comes from applying integration by parts to the divergence form of $S_{k,n}$, Proposition \ref{radial k Hessian proposition}.

Set also:
\[
I_{k,n}[u] := \frac{1}{k+1}\int_0^1 (-u)^{k+1}\ r^{n-1} dr.
\]
It is straightforward to verify that the solutions to \eqref{radial k eigenvalue} are the critical points of the functional:
\[
E_{k,n}[u] - \lambda_{k,n}^k I _{k,n}[u],
\]
acting on the space of $\alpha$-convex $u$.

By repeating the arguments for the (not necessarily radial) integer $k$-Hessian equation \cite{W94}, one can then show that:
\begin{proposition}
Suppose that $1\leq k \leq n$, and set $\alpha = \frac{n}{k}$. Then there exists a unique $\alpha$-convex function $u_{k,n}$ and a unique number $\lambda_{k,n} > 0$ which solve the $(k,n)$-eigenvalue problem \eqref{radial k eigenvalue}. Further, we have the following variational formula for $\lambda_{k,n}$ in terms of Rayleigh quotients:
\begin{equation}\label{Rayleigh quotient formula}
\lambda_{k,n}^{k} = \inf \left\{ \frac{E_{k,n}[u]}{I_{k,n}[u]}\ \middle|\ u\text{ $\alpha$-convex with }u(1) = 0, u(0) = -1\right\}.
\end{equation}
\end{proposition}

We refer to $u_{k,n}$ as the {\bf first radial $(k,n)$-eigenfunction} and $\lambda_{k,n}$ as the {\bf first radial $(k,n)$-eigenvalue.}

\subsection{Stirling's Formula and Laplace's Method}

Lastly, we recall some standard asymptotic formula and prove an increasing property for certain binomial coefficients. The first formula we will need is Stirling's formula for $\log \Gamma$ \cite[Eqn. (5.11.1)]{OLBC10}, which says that
\[
\log \Gamma(x) = (x-1) \log (x-1) - (x-1) + \frac{1}{2}\log(2\pi (x-1)) + \frac{1}{12(x-1)} + O( x^{-3}).
\]
\begin{proposition}\label{binom coeff prop}
Suppose that $(n_i)$, $(k_i)$, $1 \leq k_i \leq n_i$, are two sequences of real numbers such that $(n_i)$ increases to $\infty$ as $i\rightarrow\infty$ and:
\[
\lim_{i\rightarrow\infty} \frac{n_i}{k_i} = \alpha < \infty.
\]
Then:
\[
\lim_{i\rightarrow\infty} {n_i \choose k_i}^{\frac{1}{k_i}} = \begin{cases} \frac{\alpha^{\alpha}}{(\alpha - 1)^{\alpha -1}} & \text{ if }\alpha > 1 \\ 
1 &\text{ if }\alpha = 1.
\end{cases}
\]
\end{proposition}
\begin{proof}
This follows directly from applying Stirling's formula.
\end{proof}

We will also need the following monotonicity formula:
\begin{proposition}\label{binom prop}
For any $\alpha > 1$, the function:
\[
k \mapsto {\alpha k\choose k}^{\frac{1}{k}}
\]
is strictly increasing on $(0,\infty)$.
\end{proposition}
\begin{proof}
The result will follow from showing that:
\[
G(k) := \log {\alpha k \choose k}
\]
is strictly convex on $(0, \infty)$, since $\lim_{k\rightarrow 0} G(k) = 0$ and the secant slopes of a (strictly) convex function is (strictly) increasing; exponentiating will then prove the result.

Let $\psi_1(z) := \frac{d^2}{dz^2}\log\Gamma(z)$ denote the trigamma function, which admits the following integral representation, obtained by differentiating \cite[Eqn. (5.9.12)]{OLBC10},
\[
\psi_1(z) =\int_0^\infty \frac{t e^{-zt}}{1-e^{-t}}\,dt \qquad z>0.
\]
For $c, k >0$ and $z = ck + 1$, the change of variables $u=ct$ gives:
\begin{equation}\label{change var}
c^2\psi_1(ck+1) =\int_0^\infty \frac{u e^{-ku}}{e^{u/c}-1}\,du.
\end{equation}
A straightforward computation now gives:
\begin{align*}
G''(k)&=\alpha^2\psi_1(\alpha k+1) -\psi_1(k+1) -(\alpha-1)^2\psi_1((\alpha-1) k+1)\\
&=\int_0^\infty u e^{-ku}
\left[
\frac1{e^{u/\alpha}-1}
-\frac1{e^u-1}
-\frac1{e^{u/(\alpha-1)}-1}
\right]du,
\end{align*}
by using \eqref{change var}. It remains to show that the expression in square brackets is positive.

Fix $u>0$, and define $H(c):=\frac1{e^{u/c}-1}$ for $c > 0$. We will be done if we can show:
\[
H(a)+H(b) < H(a+b) \quad \text{ for } a,b>0.
\]
To see this, set $x = \frac{u}{c}$, so that:
\[
\frac{H(c)}{c} =\frac{1}{u}\,\frac{x}{e^x-1}.
\]
Observe then that $\frac{x}{e^x-1}$ is strictly decreasing on $(0,\infty)$; it follows that $\frac{H(c)}{c}$ is strictly increasing. Hence:
\begin{align*}
H(a) + H(b) &= a\left(\frac{H(a)}{a}\right) + b\left(\frac{H(b)}{b}\right)\\
&< a \left(\frac{H(a+b)}{a+b}\right) + b\left(\frac{H(a+b)}{a+b}\right)\\
&= H(a+b).\qedhere
\end{align*}

\end{proof}

We will also need to use the boundary form of Laplace's method \cite[Sec. 2.3(iii)]{OLBC10}; suppose that $\phi$ is a smooth, bounded function on a finite interval $(a-\e, b + \e) \subset \R$ and $(k_i)$ a sequence increasing to infinity. If $\phi$ achieves its unique maximum on $[a,b]$ at $b$, with $\phi'(b) = 0$ and $\phi''(b) < 0$, then:
\begin{equation}\label{boundary Laplace}
\int_a^b e^{k_i \phi(s)}\, ds = \frac{1}{2}\sqrt{\frac{2\pi}{k_i \abs{\phi''(b)}}}e^{k_i \phi(b)}\left(1 + O\left(k_i^{-1/2}\right)\right) .
\end{equation}

\section{Asymptotics for the first eigenvalue}

In this section, we prove Theorem \ref{limit theorem intro}; we first prove a generic lower bound, using a similar method to \cite[Eqn. (2.1)]{Le25}, and then an asymptotically sharp upper bound, assuming that $n$ and $k$ grow at a comparable rate.

\begin{proposition}\label{lower bound}
Let $1\leq k \leq n$ and set $\alpha = \frac{n}{k}$.
\begin{enumerate}[i)]
\item If $\alpha < 2$, then $\lambda_{k,n}^k \geq \frac{C_{k,n} (2-\alpha)^{k+1}}{n\cdot B\left(\frac{n}{2-\alpha}, k+1\right)}$.\\
\item If $\alpha = 2$, then $\lambda_{k,n}^k \geq \frac{C_{k,n} \cdot n^{k}}{\Gamma(k+1)}$.\\
\item If $\alpha > 2$, then $\lambda_{k,n}^k \geq \frac{C_{k,n} (\alpha-2)^{k+1}}{n\cdot B\left(\frac{2k}{\alpha-2}, k+1\right)}$.
\end{enumerate}
\end{proposition}
\begin{proof}
We will omit subscripts in the proof for notational simplicity. From \eqref{Rayleigh quotient formula}, we have:
\begin{equation}\label{a formula}
\lambda^k = \frac{E[u]}{I[u]};
\end{equation}
it will thus suffice to estimate $I[u]$ from above by a multiple of $E[u]$.

Let $r \in (0, 1)$, and consider the two cases when $\alpha\not= 2$. Applying H\"older's inequality with conjugate exponents $k+1$ and $1 + \frac{1}{k}$, we estimate:
\begin{align*}
\abs{u(r)} &= \int_r^1 \left(t^{\frac{n-k}{k+1}} u'(t) \right)\cdot t^{-\frac{n-k}{k+1}}\, dt\\
& \leq \left(\int_r^1 t^{n-k} (u'(t))^{k+1}\, dt\right)^{\frac{1}{k+1}} \left(\int_r^1 t^{1 -\alpha}\, dt \right)^{\frac{k}{k+1}}\\
& \leq \left(\frac{n(k + 1)}{C}\right)^{\frac{1}{k+1}} E[u]^{\frac{1}{k + 1}} \left(\frac{1-r^{2-\alpha}}{2-\alpha} \right)^{\frac{k}{k+1}},
\end{align*}
where in the last line, we have used \eqref{C expression for E}. It follows that:
\begin{align}
I[u] &= \frac{1}{k+1}\int_0^1 r^{n-1} \abs{u}^{k+1}(r)\, dr \leq \frac{n E[u]}{C}\int_0^1 r^{n-1} \left(\frac{1-r^{2-\alpha}}{2-\alpha}\right)^kdr
\label{another formula}
\end{align}
This integral can be evaluated in terms of the beta function, depending on if we are in cases $i)$ or $iii)$:\\

{\noindent \bf Case $i)$:} $\alpha < 2$

Making the substitution $x = r^{2-\alpha}$ directly gives:
\begin{align*}
\int_0^1 r^{n-1} \left(\frac{1-r^{2-\alpha}}{2-\alpha}\right)^k\, dr &= (2-\alpha)^{-k-1} \int_0^1 x^{\frac{n}{2-\alpha}-1} (1-x)^k\, dx\\
& = (2-\alpha)^{-k-1} B\left(\frac{n}{2-\alpha}, k+1\right).
\end{align*}
Combining this with \eqref{a formula} and \eqref{another formula} finishes this case.\\

{\noindent \bf Case $iii)$:} $\alpha > 2$

Make now the substitution $x = r^{\alpha - 2}$ to get:
\begin{align*}
\int_0^1 r^{n-1} \left(\frac{r^{2-\alpha}-1}{\alpha-2}\right)^k\, dr &= (\alpha-2)^{-k-1} \int_0^1 x^{\frac{n}{\alpha-2} - 1}(x^{-1} - 1)^k\, dx\\
& = (\alpha-2)^{-k-1} \int_0^1 x^{\frac{2k}{\alpha-2} - 1}(1 - x)^k\, dx\\
&= (\alpha - 2)^{-k-1} B\left(\frac{2k}{\alpha  -2}, k + 1\right).
\end{align*}
Again, we finish by combining with \eqref{a formula} and \eqref{another formula}.\\

Finally, we repeat the above process for case $ii)$, when $\alpha = 2$. Applying H\"older's inequality as before, we have:
\[
\abs{u(r)} \leq \left(\frac{n(k+1)}{C}\right)^{\frac{1}{k+1}} E[u]^{\frac{1}{k + 1}} \left(-\log r\right)^{\frac{k}{k+1}}.
\]
which implies:
\[
I[u] \leq \frac{n}{C}\, E[u]\int_0^1 r^{n-1} (-\log r)^k\, dr.
\]
This integral can be evaluated in terms of the $\Gamma$ function by using the substitution $x = -n\log r$, giving us:
\[
\int_0^1 r^{n-1} (-\log r)^k\, dr = \frac{1}{n^{k+1}} \int_0^\infty x^k e^{-x}\, dx = \frac{\Gamma(k+1)}{n^{k+1}}.\qedhere
\]
\end{proof}

\medskip

Before continuing, we make some preliminary observations about $r_\alpha$ \eqref{r alpha and A alpha defn} and the limiting function $v_\alpha$ \eqref{v alpha defn}, \eqref{v 2 defn}.
\begin{lemma}
For each $\alpha \geq 1$, we have $v_\alpha \in C^2([0, 1])$. It is moreover $\alpha$-convex and satisfies $v_\alpha(0) = -1$ and $v_\alpha(1) = 0$.
\end{lemma}
\begin{proof}
These properties are all straightforward to check, but we provide some computations for the reader's convenience. Start by observing that $v_\alpha$ is continuous: if $\alpha \not=2$, then
\[
v_\alpha (r_\alpha) = -e^{-\frac{\alpha}{2}} =\frac{2}{2-\alpha} \left(\frac{\alpha}{2}-1\right) e^{-\alpha/2} = A_\alpha( r_\alpha^{2-\alpha} - 1),
\]
while if $\alpha = 2$ we have $v_\alpha(r_\alpha) = -e^{-1} = A_\alpha \log r_\alpha$. That $v_\alpha(0) = -1$ and $v_\alpha(1) = 0$ are clear.

Computing the first and second derivatives of $v_\alpha$ for $r < r_\alpha$ now gives:
\begin{equation}\label{some v derivatives}
v_\alpha'(r) = -\frac{\alpha r}{r_\alpha^2} v_\alpha(r) \quad \quad v_\alpha''(r) = -\left(\frac{\alpha}{r_\alpha^2} - \frac{\alpha^2 r^2}{r_\alpha^4}\right) v_\alpha(r),
\end{equation}
which satisfy $v''_\alpha + (\alpha - 1)\frac{v'_\alpha}{r} = \frac{\alpha^2(r_\alpha^2- r^2)}{r_\alpha^4} (-v_\alpha(r))\geq 0$; if $\alpha\not =2$, then for $r_\alpha < r < 1$, we have
\[
v_\alpha'(r) = (2-\alpha) A_\alpha r^{1-\alpha} \quad\quad v_\alpha''(r) = (2-\alpha)(1-\alpha) A_\alpha r^{-\alpha}
\]
so that $v_\alpha'' + (\alpha - 1)\frac{v_\alpha'}{r} = 0$ here.
%
Taking limits now shows that $v_\alpha'$ and $v_\alpha''$ are continuous; in particular, $v_\alpha$ is $\alpha$-convex on all of $(0,1)$. The computations are similar for the case when $\alpha = 2$.
\end{proof}

\begin{lemma}\label{bad lemma}
For $\alpha \in [1, \infty)$ we have:
\begin{enumerate}[(a)]
\item the function $\alpha \mapsto r_\alpha$ is strictly increasing, with $\frac{1}{2} \leq r_\alpha < 1$; and
\item for each fixed $r\in (0, 1)$, $\alpha \mapsto v_\alpha(r)$ is strictly increasing.
\end{enumerate}
\end{lemma}
\begin{proof}
To show (a), it will suffice to show that $\frac{d r_\alpha}{d\alpha} > 0$ for $\alpha \not=2$. Using \eqref{r alpha and A alpha defn}, compute:
\[
\frac{d r_\alpha}{d \alpha} = \left( \alpha \log\left(\frac{\alpha}{2}\right) + 2-\alpha\right) \frac{r_\alpha}{\alpha(2-\alpha)^2};
\]
one then checks that $\alpha\log\left(\frac{\alpha}{2}\right) + 2-\alpha$ is convex with a minimum at $\alpha = 2$, and hence positive.

We now show (b); fix $r$ and define $f(\alpha) := v_\alpha(r)$. We have several cases, depending on the values of $r$ and $\alpha$. First, if $r \leq \frac{1}{2}$, then for any $\alpha$ we have:
\[
f(\alpha) = -e^{-\frac{\alpha r^2}{2r_\alpha^2}},
\]
so that, for $\alpha\not=2$, we have
\begin{align*}
\frac{d f}{d\alpha} &= \frac{\alpha - 2 - 2 \log \left(\frac{\alpha}{2}\right)}{(2 - \alpha)^2}\, \frac{\alpha r^2}{2r_\alpha^2}e^{-\frac{\alpha r^2}{2r_\alpha^2}}.
\end{align*}
One again checks that $\alpha - 2 - 2\log\left(\frac{\alpha}{2}\right)$ is convex with a minimum at $\alpha = 2$, and thus positive.

If $r > \frac{1}{2}$, then $f(\alpha)$ is the piecewise function:
\[
f(\alpha) = \begin{cases} A_\alpha(r^{2-\alpha} - 1) &\text{ for $\alpha$ with $r_\alpha \leq r$} \\ -e^{-\frac{\alpha r^2}{2r_\alpha^2}} &\text{ for $\alpha$ with $r_\alpha > r$.}  \end{cases}
\]
One can check that $f$ is $C^1$, so it suffices to show that it has positive derivative for $\alpha$ with $r_\alpha\not =r$ and $\alpha\not=2$. The case when $r_\alpha > r$ follows from above, so we are left to consider when $r_\alpha \leq r$. Then we have that:
\begin{align*}
\frac{d}{d\alpha}\left( A_\alpha (r^{2-\alpha} - 1)\right) 
&= A_\alpha \left[ r^{2-\alpha}\left(\frac{\alpha}{2(2-\alpha)}  - \log r \right) - \frac{\alpha}{2(2-\alpha)}\right] \\
&= \frac{e^{-\frac{\alpha}{2}}}{(2-\alpha)^2} \left((\alpha - 2 (2-\alpha) \log r)r^{2-\alpha} - \alpha\right).
\end{align*}
Consider $g(r) := (\alpha - 2(2-\alpha)\log r)r^{2-\alpha} - \alpha$ as a function of $r$ on $\left[r_\alpha, 1\right]$. Then:
\[
\frac{dg}{dr} = -(2-\alpha)^2 (1 + 2\log r)r^{1-\alpha}.
\]
If $\alpha > 2$, then $\frac{dg}{dr} < 0$ on all of $[r_\alpha, 1]$, so that $g(r) > g(1) = 0$. If $\alpha < 2$, then $\frac{dg}{dr} \geq 0$ on exactly $[r_\alpha, e^{-\frac{1}{2}}]$, so we also need to check the value of $g(r_\alpha)$:
\[
g(r_\alpha) =  \left(\alpha - 2 \log \frac{\alpha}{2}\right)\left(\frac{\alpha}{2}\right) - \alpha > 0.\qedhere
\]
%
\end{proof}

We now derive asymptotic bounds for $E_{k,n}[v_\alpha]$ and $I_{k,n}[v_\alpha]$:
\begin{proposition}\label{Laplace prop}
Suppose that $(n_i), (k_i)$, $1 \leq k_i \leq n_i$, are sequences such that $\lim_{i\rightarrow\infty} n_i = \infty$ and also:
\[
\lim_{i\rightarrow\infty} \frac{n_i}{k_i} = \alpha < \infty.
\]
Set $\alpha_i := \frac{n_i}{k_i}$. Then for sufficiently small $\e > 0$ and large $i$, we have:
\[
\frac{n_i(k_i+1)}{C_{k_i, n_i}}E_{k_i, n_i}[v_{\alpha_i}] \leq \left(\frac{1+\e}{1-\e}\right)^{n_i+3} \frac{\alpha_i \sqrt{\pi} e^{-1/2}}{\sqrt{n_i + 1}}\, 2^{k_i} e^{-\frac{\alpha_i k_i}{2}} + \left(\frac{2\e}{(1-\e)r_\alpha^{\alpha_i}}\right)^{k_i},
\]
and
\[
(k_i + 1)I_{k_i,n_i}[v_{\alpha_i}] \geq  \frac{(1-\e)\sqrt{\pi} }{\sqrt{\alpha(1+\e)} \sqrt{k_i+1}}\ r_{(1+\e)\alpha}^{(1+\e)\alpha(k_i + 1) + 1}e^{-\frac{(1+\e)\alpha(k_i+1)}{2}}.
\]
\end{proposition}
\begin{proof}

We will use Laplace's method to estimate each of these integrals. For notational convenience, set $v_{\alpha_i} =: v_i$, $E_{k_i, n_i} =: E_i$, etc. As the computations for $\alpha_i = 2$ are similar, we will also assume that $\alpha_i \not= 2$. Use \eqref{C expression for E} to write:
\begin{align*}
\frac{n_i (k_i+1)}{C_i} E_i[v_i] &= \int_0^1 r^{n_i-k_i} (v_i')^{k_i+1}\, dr\\
& = \left(\frac{\alpha_i}{r_i^2}\right)^{k_i+1}\int_0^{r_i} r^{n_i + 1} e^{-\frac{\alpha_i (k_i+1)}{2r_i^2} r^2} \, dr + \left(2 e^{-\alpha_i/2}\right)^{k_i+1}\int_{r_i}^1 r^{1 - \alpha_i}\, dr.
\end{align*} 
The second integral can be evaluated using the definition of $r_i$:
\[
\int_{r_i}^1 r^{1 - \alpha_i}\, dr = \frac{1 - r_i^{2-\alpha_i}}{2-\alpha_i} = \frac{1}{2}.
\]
 giving us:
\begin{equation}\label{dominating term}
\left(2 e^{-\alpha_i/2}\right)^{k_i+1}\int_{r_i}^1 r^{1 - \alpha_i}\, dr = 2^{k_i} e^{-\frac{\alpha_i(k_i + 1)}{2}}.
\end{equation}
To estimate the first integral, define
\[
\phi_\e(r) := \log r  - \frac{r^2}{2 (1+\e)^2 r_\alpha^2};
\]
observe that, for $i$ sufficiently large, we have:
\[
\int_0^{r_i} r^{n_i + 1} e^{-\frac{\alpha_i (k_i+1)}{2r_i^2} r^2} \, dr = \int_0^{r_i} e^{(n_i + 1)\log r -\frac{n_i+\alpha_i}{2r_i^2} r^2} \, dr  \leq  \int_0^{(1+\e)r_\alpha} e^{(n_i+1) \phi_\e(r)}\, dr.
\]
The function $\phi_\e(r)$ is smooth and strictly increasing for $r \in (0, (1+\e)r_\alpha)$, and at $r = (1+\e)r_\alpha$ satisfies
\begin{align*}
\phi_\e = \log&((1+\e) r_\alpha)-\frac{1}{2},\ \phi'_\e = 0, \text{ and } \phi''_\e = -\frac{2}{(1+\e)^2r_\alpha^2};
\end{align*}
for $\e$ sufficiently small, the boundary form of Laplace's method \eqref{boundary Laplace} then implies:
\[
\int_\e^{(1+\e)r_\alpha} e^{(n_i+1) \phi(r)}\, dr = \frac{ \sqrt{\pi}\ ((1+\e)r_\alpha)^{n_i+2}}{2\sqrt{n_i + 1}} e^{-\frac{n_i+1}{2}}\left(1 + O\left(n_i^{-1/2}\right)\right)
\]
Since $\frac{\alpha_i}{r_i^2} = \frac{2}{r_i^{\alpha_i}}$, this bounds the first integral for $i$ sufficiently large as:
\begin{align*}
\left(\frac{\alpha_i}{r_i^2}\right)^{k_i+1}\int_0^{r_i} r^{n_i + 1} e^{-\frac{\alpha_i (k_i+1)}{2r_i^2} r^2} \, dr\ \leq \left(\frac{\alpha_i}{r_i^2}\right)^{k_i+1}\left[ \e^{n_i+1} + \int_\e^{r_i} e^{(n_i+1)\phi_\e(r)}\, dr\right] &\\
\leq \frac{2^{k_i+1}}{r_i^{n_i+\alpha_i}}\left[ \e^{n_i + 1} + \frac{\sqrt{\pi}(1+\e)^{n_i+3} }{2\sqrt{n_i + 1}}r_\alpha^{n_i+2}e^{-\frac{n_i + 1}{2}}\right] &\\
\leq \left(\frac{2 \e}{(1-\e)r_\alpha^{\alpha_i}}\right)^{k_i + 1} + \left(\frac{1+\e}{1-\e}\right)^{n_i+3} \frac{\sqrt{\pi} }{\sqrt{n_i + 1}} \frac{\alpha_i}{2}\,\, 2^{k_i} e^{-\frac{n_i + 1}{2}}&.
\end{align*}
Adding this to \eqref{dominating term} and possibly shrinking $\e$ further gives:
\begin{align*}
\frac{n_i(k_i+1)}{C_i} E[v_i] \leq \left[\left(\frac{1+\e}{1-\e}\right)^{n_i+3} \frac{\sqrt{\pi} }{\sqrt{n_i + 1}} \frac{\alpha_i e^{\frac{\alpha_i - 1}{2}}}{2} + 1\right]\, 2^{k_i} e^{-\frac{\alpha_i(k_i + 1)}{2}} + \left(\frac{2\e}{(1-\e)r_\alpha^{\alpha_i}}\right)^{k_i},
\end{align*}
from which we deduce the desired upperbound for the energy for all large $i$.

\medskip

We now estimate $I_i[v_i]$, using again Laplace's method. Set
\[
\psi_\e(r) := (1+\e)\alpha\log r + \log \abs{v_{(1+\e)\alpha}}.
\]
By Lemma \ref{bad lemma}, we have:
\[
\alpha_i \log r + \log \abs{v_i} \geq \psi_\e(r)
\]
for all $i$ sufficiently large. 
It follows that:
\[
(k_i+1)I_i[v_i] = \int_0^1 r^{n_i-1}\abs{v_i}^{k_i+1}\, dr \geq \int_0^1 e^{(k_i+1)\left( \alpha_i \log r + \log \abs{v_i}\right)}\, dr \geq \int_0^1 e^{(k_i+1)\psi_\e(r)}\, dr,
\]
using $\frac{n_i - 1}{k_i + 1} \leq \alpha_i$. We claim that the unique maximum of $\psi_\e$ occurs at $r = r_{(1+\e)\alpha}$ -- for convenience, in the computations below, we set $\beta_\e := (1+\e) \alpha$, $s_\e := r_{\beta_\e}$, and $w_\e := v_{\beta_\e}$. To see the claim, compute:
\[
\psi'_\e(r) = \frac{\beta_\e}{r} + \frac{w_{\e}'(r)}{w_{\e}(r)}.
\]
For $r \leq s_{\e}$, recall from \eqref{some v derivatives} that:
\[
w'_\e(r) = -\frac{\beta_\e r}{s_\e^2} w_\e(r) \implies \psi'_\e(r) = \beta_\e r \left(\frac{1}{r^2} - \frac{1}{s_\e^2}\right) \geq 0,
\]
with equality only at $r = s_\e$. For $s_\e < r < 1$ and $\beta_\e\not= 2$ (which we can assume by shrinking $\e$ slightly if necessary), we have $w_\e(r) = A_{\beta_\e} (r^{2-\beta_\e}-1)$ so that:
\[
\psi'_\e(r) = \frac{\beta_\e}{r} + \frac{(2-\beta_\e) r^{1-\beta_\e}}{r^{2-\beta_\e} - 1} = \frac{2 r^{2-\beta_\e} - \beta_\e}{r (r^{2-\beta_\e} - 1)} < 0,
\]
since $s_\e = \left(\frac{\beta_\e}{2}\right)^{\frac{1}{2-\beta_\e}} < r < 1$. The claim is shown. 

We now have:
\[
\psi_\e(s_\e) = \beta_\e \log s_\e - \frac{\beta_\e}{2},
\]
and
\[
\psi''_\e(s_\e) = -\frac{\beta_\e}{s_\e^2} + \frac{w_\e'' w_\e - (w_\e')^2}{w_\e^2}(s_\e) = -\frac{\beta_\e}{s_\e^2}  - \frac{2}{s_\e^{\beta_\e}} = -\frac{2\beta_\e}{s_\e^{2}}.
\]
Since the restriction of $\psi_\e$ to both $(\e, s_\e)$ and $(s_\e, 1-\e)$ admit smooth extensions past the endpoints, we can apply Laplace's method \eqref{boundary Laplace} to both halves to get:
\begin{align*}
(k_i+1) I_i[v_i] &\geq \int_\e^{s_\e} e^{(k_i+1)\psi_\e(r)}\, dr + \int_{s_\e}^{1-\e} e^{(k_i+1)\psi_\e(r)}\, dr\\
&\geq  \frac{(1-\e)\sqrt{\pi}}{\sqrt{\beta_\e (k_i+1)}}s_\e^{\beta_\e(k_i + 1) + 1} e^{-\frac{\beta_\e(k_i+1)}{2}},
\end{align*}
as desired.
\end{proof}

We now prove Theorem \ref{limit theorem intro}:
\begin{proof}[Proof of Theorem \ref{limit theorem intro}]
Set $\alpha_i := \frac{n_i}{k_i}$ and fix $\e \in \left(0,\frac{1}{2}\right)$ small. From the Rayleigh quotient formula \eqref{Rayleigh quotient formula}, we have:
\[
\lambda_{k_i, n_i}^{k_i} \leq \frac{E_{k_i,n_i}[v_{\alpha_i}]}{I_{k_i,n_i}[v_{\alpha_i}]}.
\]
Applying Proposition \ref{Laplace prop}, using the inequality $a^k + b^k \leq (a+b)^k$ for $a,b > 0$ and $k \geq 1$, we then have:
\begin{align*}
\lambda_{k_i, n_i} &\leq \left(\frac{C_{k_i,n_i}}{n_i}\right)^{\frac{1}{k_i}} \left( \left(\frac{1+\e}{1-\e}\right)^{4} \frac{\alpha_i \sqrt{\alpha(k_i+1)}}{\sqrt{n_i + 1}}\frac{e^{\frac{(1+\e)\alpha-1}{2}}}{r_{(1+\e)\alpha}^{(1+\e)\alpha + 1}}\right)^{\frac{1}{k_i}}\left(\frac{1+\e}{1-\e}\right)^{\alpha_i}\frac{2 e^{-\frac{\alpha_i}{2}}}{r_{(1+\e)\alpha}^{(1+\e)\alpha}  e^{-\frac{(1+\e)\alpha}{2}} }\\
&+ \left(\frac{C_{k_i,n_i}}{n_i}\right)^{\frac{1}{k_i}} \left(\frac{\sqrt{\alpha(1+\e)} \sqrt{k_i+1} }{(1-\e)\sqrt{\pi} } \frac{e^{\frac{(1+\e)\alpha}{2}}}{r_{(1+\e)\alpha}^{(1+\e)\alpha + 1}}\right)^{\frac{1}{k_i}} \frac{e^{\frac{(1+\e)\alpha}{2}}}{r_{(1+\e)\alpha}^{(1+\e)\alpha}}\, \frac{2\e}{(1-\e)r_\alpha^{\alpha_i}}.
\end{align*}
Take the limit as $i\rightarrow\infty$: by Proposition \ref{binom coeff prop}, the factor $\left(\frac{C_{k_i,n_i}}{n_i}\right)^{\frac{1}{k_i}} \rightarrow \frac{\alpha^\alpha}{(\alpha-1)^{\alpha-1}}$, while the large middle factors in both terms tend to 1. Thus:
\[
\limsup_{i\rightarrow\infty} \lambda_{k_i, n_i} \leq \frac{\alpha^{\alpha}}{(\alpha - 1)^{\alpha -1}}\left(\frac{1+\e}{1-\e}\right)^{\alpha} \frac{2}{r_{(1+\e)\alpha}^{(1+\e)\alpha}}  e^{\frac{(1 +\e)\alpha}{2} -\frac{\alpha}{2}} + K\e,
\]
for a finite $K$, independent of $\e$. Letting $\e\rightarrow 0$ now gives the upper bound:
\[
\limsup_{i\rightarrow\infty} \lambda_{k_i, n_i} \leq \frac{\alpha^{\alpha}}{(\alpha - 1)^{\alpha -1}} \frac{2}{r_{\alpha}^{\alpha}} = \frac{\alpha^{\alpha+1}}{(\alpha-1)^{\alpha-1}} \left(\frac{2}{\alpha}\right)^{\frac{2}{2-\alpha}}.
\]

For the lower bound, we use Proposition \ref{lower bound}, using Stirling's formula in the standard way to compute the limit of the beta and Gamma functions appearing in the denominator. By standard reasoning, it suffices to consider sequences all $\alpha_i < 2$, all $\alpha_i > 2$, or all $\alpha_i = 2$. If all $\alpha_i < 2$, Proposition \ref{lower bound} implies that
\[
\lambda_{k_i,n_i} \geq \frac{C_{k_i,n_i}^{\frac{1}{k_i}} (2-\alpha_i)^{1+\frac{1}{k_i}}}{n_i^{\frac{1}{k_i}}\,B\left(\frac{n_i}{2-\alpha_i}, k_i+1\right)^{\frac{1}{k_i}}}.
\]
Now Stirling's formula, equation \eqref{beta to binom}, and Proposition \ref{binom coeff prop} can be used to show that:
\[
\lim_{i\rightarrow\infty}B\left(\frac{n_i}{2-\alpha_i}, k_i+1\right)^{\frac{1}{k_i}} = \lim_{i\rightarrow\infty} {\frac{2}{2-\alpha_i} k_i \choose k_i}^{-\frac{1}{k_i}} = \frac{2-\alpha}{\alpha}\left(\frac{\alpha}{2}\right)^{\frac{2}{2-\alpha}}.
\]
Using Proposition \ref{binom coeff prop} again, we have the desired lower bound:
\[
\liminf_{i\rightarrow\infty} \lambda_{k_i, n_i} \geq \frac{\alpha^{\alpha+1}}{(\alpha-1)^{\alpha-1}}\left(\frac{2}{\alpha}\right)^{\frac{2}{2-\alpha}}.\qedhere
\]
The other cases are similar -- if $\alpha_i >  2$, then Proposition \ref{lower bound} and Stirling's formula give:
\[
\liminf_{i\rightarrow\infty} \lambda_{k_i,n_i} \geq \lim_{i\rightarrow\infty} \frac{C_{k_i,n_i}^{\frac{1}{k_i}} (\alpha_i-2)^{1 + \frac{1}{k_i}}}{n_i^{\frac{1}{k_i}}\cdot B\left(\frac{2k_i}{\alpha_i-2}, k_i+1\right)^{\frac{1}{k_i}}} = \left(\frac{\alpha^\alpha}{(\alpha-1)^{\alpha-1}}\right) \cdot\alpha\left(\frac{\alpha}{2}\right)^{\frac{2}{\alpha-2}},
\]
while if all $\alpha_i = 2$, we have:
\[
\liminf_{i\rightarrow\infty} \lambda_{k_i,n_i} \geq \lim_{i\rightarrow\infty}\frac{C_{k_i,n_i}^{\frac{1}{k_i}} \cdot n_i}{\Gamma(k_i+1)^{\frac{1}{k_i}}} = 8e.
\]
\end{proof}

\section{Monotonicity}

In this section, we prove Theorem \ref{monotone theorem intro}. Given $1\leq k\leq n$, we continue to use $\alpha := \frac{n}{k}$ and $C_{k,n} := {n\choose k}$. Define now:
\begin{equation}\label{r k defn}
r_k^2 := \frac{\alpha C_{k,n}^{\frac{1}{k}}}{\lambda_{k,n}}.
\end{equation}
Observe that $r_k \rightarrow r_\alpha$ as $k\rightarrow\infty$, by Theorem \ref{limit theorem intro}, where $r_\alpha$ is defined in \eqref{r alpha and A alpha defn}. We start with the following observation:
\begin{proposition}\label{ratio ODE}
Let $1 \leq k \leq n$ and $\alpha := \frac{n}{k}$. Let $u := u_{k,n}$ be the first radial $(k,n)$-eigenfunction, and define:
\[
w_{k,n} := \frac{\alpha u'}{r\left(u'' + (\alpha-1) \frac{u'}{r}\right)}.
\]
Then $w_{k,n}$ satisfies:
\begin{equation}\label{w equation}
w'_{k,n} = \frac{n}{r}\left(1 - w_{k,n} + \frac{r^2}{r_k^2} w^{1 + \frac{1}{k}}_{k,n}\right).
\end{equation}
\end{proposition}
\begin{proof}
For convenience, we omit subscripts in the arguments below. Set $s := u'' + (\alpha - 1)\frac{u'}{r}$, and start by computing the derivative of $w$:
\[
\frac{s^2}{\alpha} w' = \frac{s u''}{r} - \frac{s u'}{r^2}- \frac{s' u'}{r}
\]
On the other hand, by differentiating the eigenvalue equation \eqref{radial k eigenvalue}, we have:
\begin{align*}
\frac{C}{\alpha} &\left[ \frac{(k-1)(u')^{k-2} u'' s}{r^{k-1}} - \frac{(k-1)(u')^{k-1} s}{r^k} + \left(\frac{u'}{r}\right)^{k-1} s'\right]  = -k (-\lambda u)^{k-1} \lambda u'
\end{align*}
Simplify the right-hand side using \eqref{radial k eigenvalue}:
\begin{align*}
\frac{C}{\alpha} &\left[ \frac{(k-1)(u')^{k-2} u'' s}{r^{k-1}} - \frac{(k-1)(u')^{k-1} s}{r^k} + \left(\frac{u'}{r}\right)^{k-1} s'\right] = -k \lambda u' \left(\frac{C}{\alpha}\left(\frac{u'}{r}\right)^{k-1} s\right)^{\frac{k-1}{k}}.
\end{align*}
Now divide by $\left(\frac{u'}{r}\right)^{k-2}$ and use the above expression for $\frac{s^2}{\alpha} w'$ to get
\begin{align*}
\frac{C}{\alpha} &\left[ \frac{ks}{r}\left(u'' - \frac{u'}{r}\right) - \frac{s^2}{\alpha} w'\right] = -k \lambda \left(\frac{C}{\alpha}\right)^{\frac{k-1}{k}}\frac{(u')^{\frac{k+1}{k}}}{r^{\frac{1}{k}}} s^{\frac{k-1}{k}}.
\end{align*}
Divide out by $\frac{C}{\alpha}$ and rewrite this as:
\[
\frac{n s}{r}\left(s - \alpha \frac{u'}{r} \right) + n \lambda \left(\frac{C}{\alpha}\right)^{-\frac{1}{k}}\frac{(u')^{\frac{k+1}{k}}}{r^{\frac{1}{k}}} s^{\frac{k-1}{k}} = s^2 w'.
\]
Dividing by $s^2$ and simplifying now gives:
\[
\frac{n}{r}\left(1 - w  + \frac{\lambda r^2}{\alpha C^{\frac{1}{k}}} w^{1 + \frac{1}{k}}\right) = w'.
\]
We finish by using the definition of $r_k$, \eqref{r k defn}.
\end{proof}

The key observation in this section is that $w_{k,n}$ is bounded from below by $1$; this lower bound will immediately imply Theorem \ref{monotone theorem intro}.

\begin{proposition}\label{lower w bound}
We have the following lower bound:
\[
w_{k,n} \geq 1.
\]
In particular, we have:
\[
\frac{C_{k,n}^{\frac{1}{k}}}{\alpha} \left(u''_{k,n} + (\alpha - 1)\frac{u'_{k,n}}{r}\right) \leq -\lambda_{k,n} u_{k,n} \leq C_{k,n}^{\frac{1}{k}} \frac{u'_{k,n}}{r}.
\]
\end{proposition}
\begin{proof}
Observe that Proposition \ref{ratio ODE} implies that $w'_{k,n} \geq \frac{n}{r}(1-w_{k,n})$. Rewrite this as:
\[
0 \leq  r^{n} w'_{k,n} - n r^{n-1}(1 - w_{k,n}) = -\frac{d}{dr}\left[ r^n (1 - w_{k,n}) \right].
\]
It follows that $r^n(w_{k,n} - 1)$ is monotone increasing in $r$; since this function is $0$ at $r = 0$, we see that $w_{k,n} \geq 1$. \\

This implies now that $u'' + (\alpha - 1) \frac{u'}{r} \leq \alpha\frac{u'}{r}$; the second result follows from applying this to the eigenvalue equation \eqref{radial k eigenvalue}.
\end{proof}

We now give the:
\begin{proof}[Proof of Theorem \ref{monotone theorem intro}]
Let $n = \alpha k$ and $m = \alpha \ell$. We bound $S_{k,n}(u_{\ell, m})$ by using the $(\ell, m)$-eigenvalue equation \eqref{radial k eigenvalue} and then Proposition \ref{lower w bound} as:
\begin{align*}
\frac{C_{k,n}}{\alpha} \left(\frac{u_{\ell,m}'}{r}\right)^{k-1}\left(u''_{\ell,m} + (\alpha-1)\frac{u'_{\ell, m}}{r}\right) &= \frac{C_{k,n}}{C_{\ell,m}} \frac{(-\lambda_{\ell, m} u_{\ell, m})^{\ell}}{(u'_{\ell,m} / r)^{\ell - k}}\\
&\leq C_{k,n}\, C_{\ell, m}^{-\frac{k}{\ell}} (-\lambda_{\ell, m} u_{\ell, m})^{k},
\end{align*}
so that:
\[
S_{k,n}(u_{\ell,m}) \leq \left(C_{k,n}^{\frac{1}{k}}C_{\ell, m}^{-\frac{1}{\ell} } \lambda_{\ell, m}\right)^k (-u_{\ell, m})^k.
\]
The Rayleigh quotient formulation \eqref{Rayleigh quotient formula} of $\lambda_{k,n}$ now implies that:
\[
\lambda_{k,n}^{k} \leq \frac{E_{k,n}[u_{\ell, m}]}{I_{k,n}[u_{\ell,m}]} \leq \left(C_{k,n}^{\frac{1}{k}}C_{\ell, m}^{-\frac{1}{\ell} } \lambda_{\ell, m}\right)^k,
\]
from which the result follows.
\end{proof}

\section{Convergence of the Eigenfunctions}

We now show a partial upper bound for $w_{k,n}$, for small enough $\alpha$; from this, we will be able to deduce convergence of the eigenfunctions.

\begin{proposition}\label{upper w bound}
Suppose either that
\begin{enumerate}[(a)]
\item $\alpha \leq 2$, or
\item $\alpha < 2e$ and $k$ is sufficiently large (depending only on $\alpha$).
\end{enumerate} 
Then:
\begin{equation}\label{upper equation}
w_{k,n}(r) \leq \frac{r_k^2}{r^2_k - r^2}\quad\text{for all $r \in [0, r_k)$.}
\end{equation}
\end{proposition}
\noindent Note that, by Theorem \ref{monotone theorem intro}, $\frac{r_k^2}{r_k^2 - r^2} \leq \frac{r_\alpha^2}{r_\alpha^2 - r^2}$ for all $r \leq r_\alpha$; thus, \eqref{upper equation} implies the bound \eqref{upper w bound intro} from the Introduction.

\begin{proof}
Define $\rho(r) := \frac{r^2}{r_k^2}$, so that $\rho \in [0, 1)$. We claim that
\[
f(r) := \frac{r_k^2}{r_k^2 - r^2} = \frac{1}{1-\rho}
\]
is a supersolution to \eqref{w equation} on the range $[0, r_k)$, i.e. that:
\begin{equation}\label{supersolution claim}
\frac{df}{dr} \geq \frac{n}{r}\left(1 - f + \rho f^{1+\frac{1}{k}}\right)\quad \text{ for all $r\in[0,r_k)$}.
\end{equation}
Assuming the claim, we can finish by an application of Chaplygin comparison, which proceeds as follows. Define $d := w - f$; note that $d(0) = 0$. Suppose for the sake of a contradiction that $r_0 \in (0, r_k)$ is such that $d(r_0) > 0$. By continuity, there exists $0 \leq a < r_0 < b\leq r_k$ such that $d(a) = 0$ and $d(r) > 0$ for all $r\in (a,b)$. 

Now define:
\[
p := \frac{n \rho}{r}\frac{w^{1 + \frac{1}{k}}_{k,n} - f^{1+\frac{1}{k}}} {w_{k,n} - f}
\]
for $w \not= f$, and extend it continuously to all of $[0, r_k)$. Then, by the claim \eqref{supersolution claim} and $d > 0$ on $(a,b)$, we have:
\[
d' \leq p\, d\quad \text{ for all $r\in [a,b)$}.
\]
But then Gr\"onwall's inequality implies that $d(r_0) \leq 0$, a contradiction.

We are left to show \eqref{supersolution claim}, which is equivalent to showing that
\[
\frac{(1-\rho)^2}{\rho}\left[1 - f + \rho f^{1+\frac{1}{k}} \right]\leq \frac{(1-\rho)^2}{\rho}\frac{r}{n} \frac{df}{dr} = \frac{2}{n}.
\]
The left-hand side can be simplified to give the following function of $\rho$:
\[
h(\rho) := \rho - 1 + (1-\rho)^{1-\frac{1}{k}}.
\]
We compute the maximum of $h$ on $[0, 1]$; the exceptional case is $(a)$ when $k = 1$, in which case $\alpha = n \leq 2$ and \eqref{supersolution claim} is clear. Otherwise, the maximum of $h$ is achieved at $\rho_0\in [0, 1)$ solving
\[
1 - \left(1-\frac{1}{k}\right)(1-\rho_0)^{-\frac{1}{k}} = 0 \implies 1-\rho_0 = \left(1 - \frac{1}{k}\right)^k.
\]
This means that the maximum value is:
\[
h(\rho_0) = \frac{1}{k}\left(1 - \frac{1}{k}\right)^{k-1}.
\]
Now, in case $(a)$, we have $\frac{1}{k} \leq \frac{2}{n}$, and \eqref{supersolution claim} is immediate. If we only have $\alpha < 2e$, then we use the fact that $\left(1-\frac{1}{k}\right)^{k-1}$ decreases towards $e^{-1} < \frac{2}{\alpha}$, so that we can choose $k$ sufficiently large so that $h(\rho_0) < \frac{2}{n}$, again showing \eqref{supersolution claim}.
\end{proof}

We now translate the bounds on $w_{k,n}$ into bounds on $u_{k,n}$; we will actually get a bound on the derivative as well.
\begin{proposition}\label{entire lower u bound}
Assume again that either
\begin{enumerate}[(a)]
\item $\alpha \leq 2$, or
\item $\alpha < 2e$ and $k$ is sufficiently large (depending only on $\alpha$).
\end{enumerate} 
Define another piecewise function:
\[
v_{k, \alpha}(r)  := \begin{cases} - e^{-\frac{\alpha r^2}{2r_k^2}} &\text{ for } 0 \leq r\leq r_k \\ B_{k} r^{2-\alpha} - C_k &\text{ for }r_k < r, \end{cases}
\]
where $B_{k}, C_k$ are chosen such that $v_{k, \alpha} \in C^2(0,1)$. Then we have both:
\begin{equation}\label{entire lower u bound equation}
u_{k,n} \geq v_{k, \alpha}\quad\text{ and }\quad u_{k,n}'\geq v_{k,\alpha}' \quad \text{ on all of }[0,1].
\end{equation}
\end{proposition}
\begin{proof}
By Proposition \ref{lower w bound}, we have:
\[
u_{k,n}' \geq -\frac{\alpha r}{r_k^2} u_{k,n}.
\]
On the other hand, $v_{k,\alpha}$  satisfies $-\frac{\alpha r}{r_k^2} v_{k,\alpha} = v_{k,\alpha}'$ on the interval $(0, r_k)$. Since $v_{k,\alpha}(0) = u_{k,n}(0) = -1$, Gr\"onwall's inequality implies \eqref{entire lower u bound equation} on $[0, r_k]$. Observe that this holds for any $\alpha \geq 1$.

For the other interval, $[r_k, 1)$, assuming that it's non-empty, we have from the definition of $\alpha$-convexity that
\begin{equation}\label{alpha conv condition here}
u_{k,n}'' + (\alpha - 1)\frac{u_{k,n}'}{r} \geq 0 = v_{k,\alpha}'' + (\alpha - 1)\frac{v_{k,\alpha}'}{r}.
\end{equation}
Since $u_{k,n}(r_k) \geq v_{k,\alpha}(r_k)$, we will thus be able to conclude that $u_{k,n} \geq v_{k,\alpha}$ on all of $[0,1]$ if we can show that
\begin{equation}\label{temp claim here}
u'_{k,n}(r_k) \geq v'_{k,\alpha}(r_k),
\end{equation}
by Gr\"onwall. Actually, \eqref{temp claim here} would also imply that $u'_{k,n} \geq v'_{k,\alpha}$ on $[r_k, 1]$ as well, by rearranging \eqref{alpha conv condition here} and applying Gr\"onwall starting from $r = r_k$.

To show \eqref{temp claim here}, we will actually show that $u'_{k,n} \geq v'_{k,\alpha}$ on $[0, r_k]$, which will complete the proposition. Observe:
\[
\frac{\alpha}{r} \left(\frac{1}{\frac{d}{dr}\log v'_{k,\alpha} + \frac{\alpha - 1}{r}}\right) = \frac{\alpha v'_{k,\alpha}}{r (v_{k,\alpha}'' + (\alpha - 1)\frac{v'_{k,\alpha}}{r})} = \frac{r_k^2}{r_k^2 - r^2}
\]
on $(0, r_k)$. Thus, Proposition \ref{upper w bound} implies that:
\[
\frac{\alpha}{r} \left(\frac{1}{\frac{d}{dr}\log u'_{k,n} + \frac{\alpha - 1}{r}}\right) = w_{k,n}  \leq \frac{\alpha}{r} \left(\frac{1}{\frac{d}{dr}\log v'_{k,\alpha} + \frac{\alpha - 1}{r}}\right)
\]
on $(0, r_k)$. Rearranging gives:
\begin{equation}\label{log comparison}
\frac{d}{dr} \log v'_{k,\alpha} \leq \frac{d}{dr}\log u'_{k,n}.
\end{equation}
The eigenvalue equation \eqref{radial k eigenvalue} implies that $\lim_{r\rightarrow 0} u''_{k,n} = C_{k,n}^{-\frac{1}{k}} \lambda_{k,n}$, so now:
\[
 \lim_{r\rightarrow 0} \frac{u'_{k,n}}{v'_{k,\alpha}} = \lim_{r\rightarrow0}\frac{u''_{k,n}}{v''_{k,\alpha}} = \frac{\lambda_{k,n} r_k^2}{\alpha C_{k,n}^{\frac{1}{k}}} = 1.
\]
Integrating both sides of \eqref{log comparison} and simplifying then finishes.
\end{proof}

We can now prove Theorem \ref{eigenfunction theorem intro}

\begin{proof}[Proof of Theorem \ref{eigenfunction theorem intro}]

First, observe that the lower-bound from Proposition \ref{lower w bound} can be rewritten as:
\[
\frac{d}{dr} \log \frac{u'_{k,n}}{r} \leq 0 \quad \text{ for all $r\in(0,1)$}.
\]
Since \eqref{radial k eigenvalue} implies that $\lim_{r\rightarrow0} \frac{u'}{r} = C_{k,n}^{-\frac{1}{k}}\lambda_{k,n}$, Theorem \ref{monotone theorem intro} implies that:
\begin{equation}\label{something to reference}
0 \leq \frac{u'_{k,n}}{r} \leq 
\frac{\alpha}{r_\alpha^2},\quad \text{ for all $k\geq 1$}.
\end{equation}
Using the definition of $\alpha$-convexity and Proposition \ref{lower w bound} again, we now see:
\[
(1-\alpha)\frac{u'_{k,n}}{r} \leq u''_{k,n} \leq \frac{u'_{k,n}}{r},
\]
so that $u''_{k,n}$ is also uniformly bounded. Fix $\beta\in(0,1)$; then any subsequence of $\{u_{k,\alpha k}\}_{k\geq 1}$ has a $C^{1,\beta}$-convergent subsequence. We choose such a subsequence, and write $u_\infty$ for the limit.

We now identify $u_\infty$. First, Proposition \ref{upper w bound} and Proposition \ref{lower w bound} imply:
\begin{equation}\label{something 2}
u''_{k,n} + (\alpha - 1)\frac{u'_{k,n}}{r} \geq \frac{r_\alpha^2 - r^2}{r_\alpha^2}\frac{\alpha u'_{k,n}}{r} \geq \frac{r_\alpha^2 - r^2}{r_\alpha^2} \frac{\alpha (-\lambda_{k,n}u_{k,n})}{C_{k,n}^{1/k}}
\end{equation}
on $[0, r_\alpha)$. Now integrating the derivative bound in \eqref{entire lower u bound equation} from $r$ to $1$ gives:
\[
u_{k,n}(r) \leq v_{k,\alpha}(r) - v_{k,\alpha}(1).
\]
As $k\rightarrow\infty$, it is easy to see that the $v_{k,\alpha}$ increase uniformly to $v_\alpha$, which is continuous, strictly negative on $[0, r_\alpha)$, and satisfies $v_{\alpha}(1) = 0$; hence, on any compact subinterval of $[0, r_\alpha)$, $u_{k,n}$ is uniformly bounded away from 0 for all $k$ sufficiently large. 
%
Thus, \eqref{something 2} implies:
\[
\left(u''_{k,n} + (\alpha - 1)\frac{u'_{k,n}}{r}\right)^{\frac{1}{k}} \rightarrow 1.
\]
By then taking the limit of the $k^\text{th}$-root of the eigenvalue equation \eqref{radial k eigenvalue}, we see that $u_\infty$ satisfies:
\[
\frac{u_\infty'}{r} = -\frac{\alpha}{r_\alpha^2} u_\infty.
\]
Since $u_\infty(0) = -1$, $u_\infty = v_\alpha$ on $[0, r_\alpha)$. Since the convergence is known to be $C^{1,\beta}$ on all of $[0, 1]$, we conclude that our subsequence converges to $v_\alpha$ in $C^{1,\beta}([0, r_\alpha])$, including the endpoints.

Fix $\e > 0$. For convenience, we suppose that $\alpha < 2$ and set $\ti{v}(r) := (A_\alpha - \e)(r^{2-\alpha} - 1)$. From the above, we have
\[
u_{k,\alpha k}(r_\alpha) < \ti{v}(r_\alpha)\quad\text{ and }\quad u_{k,\alpha k}(1) = \ti{v}(1)
\]
for all $k$ sufficiently large. By the Sturm comparison theorem, applied to the $\alpha$-convexity condition $u'' + (\alpha - 1)\frac{u'}{r} \geq 0$, we see that $u_{k,\alpha k} \leq \ti{v}$ on $[r_\alpha, 1]$ for $k$ large. Since $\e$ was arbitrary, it now follows from Proposition \ref{entire lower u bound} that $u_\infty = v_\alpha$ on $[r_\alpha, 1]$ as well. 

Since our original subsequence was arbitrary, we conclude that the full sequence $\{u_{k, \alpha k}\}$ converges to $v_\alpha$ in $C^{1,\beta}([0,1])$.

For $2\leq \alpha < 2e$, the argument is the same after making the obvious changes to $\ti{v}$.
\end{proof}

\section{Complex Case}

We illustrate how our results can be applied to the complex setting. Briefly, recall that for $u \in C^\infty(B_1^{2n})$ (thinking of $B_1^{2n}\subset \R^{2n} \cong \C^n$), and integer $1\leq k\leq n$, the {\bf complex $k$-Hessian operator} acting on $u$ is defined to be:
\[
H_{k,n}(u) := \sigma_{k}(\lambda^\C),
\]
where $\lambda^\C = (\lambda_1^\C, \ldots, \lambda_n^\C)$ are the eigenvalues of the complex Hessian of $u$. Similar to Proposition \ref{radial k Hessian proposition}, we have:
\begin{proposition}\label{complex radial k Hessian proposition}
Suppose that $u$ is a smooth, radially symmetric function on $B_1^{2n}\subset \C^n$. Then:
\[
H_{k, n}(u) = \frac{1}{2^{k}} \frac{C_{k,n}}{C_{k,2n}} S_{k, 2n}(u),
\]
where we write $C_{k,n} = {n \choose k}$. 
\end{proposition}
\begin{proof}
Let $z_i = x_i + \sqrt{-1} y_i$ be complex coordinates on $\C^n$, so that $r^2 = \sum_{i=1}^n z_i\ov{z}_i$. Then the complex Hessian of $u$ can be computed as:
\[
u_{i\ov{j}} := \frac{\p^2 u}{\p z_i \pbar z_j} = \frac{1}{2}\left[\frac{u'}{r}\delta_{i\ov{j}} + \frac{1}{2}\left(u'' - \frac{u'}{r}\right)\frac{\ov{z}_i z_j}{r^2}\right]
\]
which has eigenvalues $\frac{1}{4}\left(u'' + \frac{u'}{r}\right)$, with multiplicity 1, and $\frac{u'}{2r}$, with multiplicity $n-1$. Then we see that:
\begin{align*}
H_{k, n}(u) &= {n-1 \choose k-1}\left(\frac{u'}{2r}\right)^{k-1}\cdot \frac{1}{4}\left(u'' + \frac{u'}{r}\right) + {n-1 \choose k}\left(\frac{u'}{2r}\right)^{k}\\
&= \frac{1}{2^{k+1}}{n-1 \choose k-1} \left(\frac{u'}{r}\right)^{k-1} \left[u'' + \frac{2n-k}{k}\frac{u'}{r} \right].
\end{align*}
The proposition follows from rewriting the coefficients out front.
\end{proof}

It follows that $H_{k,n}(u_{k,2n}) = \frac{1}{2^{k}} \frac{C_{k, n}}{C_{k, 2n}}(-\lambda_{k,2n} u_{k,2n})^k$, so that
\[
\lambda_{k,n}^\C := \frac{1}{2}\frac{C_{k,n}^{1/k}}{C_{k,2n}^{1/k}} \lambda_{k, 2n}
\]
is the first {\bf complex $(k,n)$-Hessian eigenvalue.} It follows from Proposition \ref{binom coeff prop} and Theorem \ref{limit theorem intro} that, for any $\alpha > 1$, we have:
\[
\lim_{k\rightarrow\infty} \lambda_{k, \alpha k}^\C = \frac{\alpha^{\alpha+1}}{(\alpha-1)^{\alpha - 1}}\left(\frac{1}{\alpha}\right)^{\frac{1}{1-\alpha}},
\]
while for $\alpha = 1$ (which corresponds to the case of the complex Monge-Amp\`ere operator) we have:
\[
\lim_{k\rightarrow\infty} \lambda_{k, k}^\C = e.
\]

%

\section{Numerical Results and Further Questions}

In addition to the above results, we were able to numerically solve \eqref{radial k eigenvalue} for large values of $k$ and $n$ by utilizing standard ODE solvers and incrementing in $n$. Briefly, we start by fixing $\alpha$ and solving the ODE for $n = 1$ by integrating by adaptive collocation (using \texttt{scipy.integrate.solve\_bvp}), refining the mesh until the relative residual meets tolerance. A bit more precisely, the eigenfunction mesh is composed of $C^1$ piecewise-cubics, and the collocation and boundary equations are solved together by a Newton solver. Each refinement step selects intervals whose post-convergence residual exceeds tolerance, after which new nodes are initialized by interpolating the previous solution and the full system is re-solved globally. The eigenvalue $\lambda$ is a single global unknown, updated simultaneously with the node values and closed by a normalization condition.

The mesh is initiated as a uniform grid with tolerance $10^{-8}$, and the initial guess we use is $u\approx r^{2}-1$. At $r = 0$, we impose the constraint $u''(0) = C^{-1/k} \lambda$ directly, to avoid dealing with the indeterminate expression $\frac{u'}{r}$. To avoid overflow, binomial coefficients are computed in logarithmic form and combined by a log-sum-exp before exponentiation.

To progress, we use continuation: walking from small $n$ to $n_{\max}$, each converged solution warm-starts the next (mesh and $\lambda$ guess). This reduces iterations in the large-$n$ regime where cold starts fail. A failed step bisects the interval $[n_{\mathrm{prev}},n]$ and inserts the midpoint, recursing until convergence or $n-n_{\mathrm{prev}} < \delta_{\min}$ (default $0.1$); the last converged solution is then reported.

For each $\alpha$, we record $\{(n,\, \lambda_{k,n},\,u_{k,n})\}$ with each eigenfunction sampled on a $500$-point grid.  We include several figures on the pages below to illustrate our outputs. Full code can be found on the second author's Github page, at \url{https://github.com/AbCoding/PXML_k-Hessian}.

\raggedbottom

\begin{figure}[H]
   \begin{minipage}{0.42\textwidth}
     \centering
     \includegraphics[width=\linewidth]{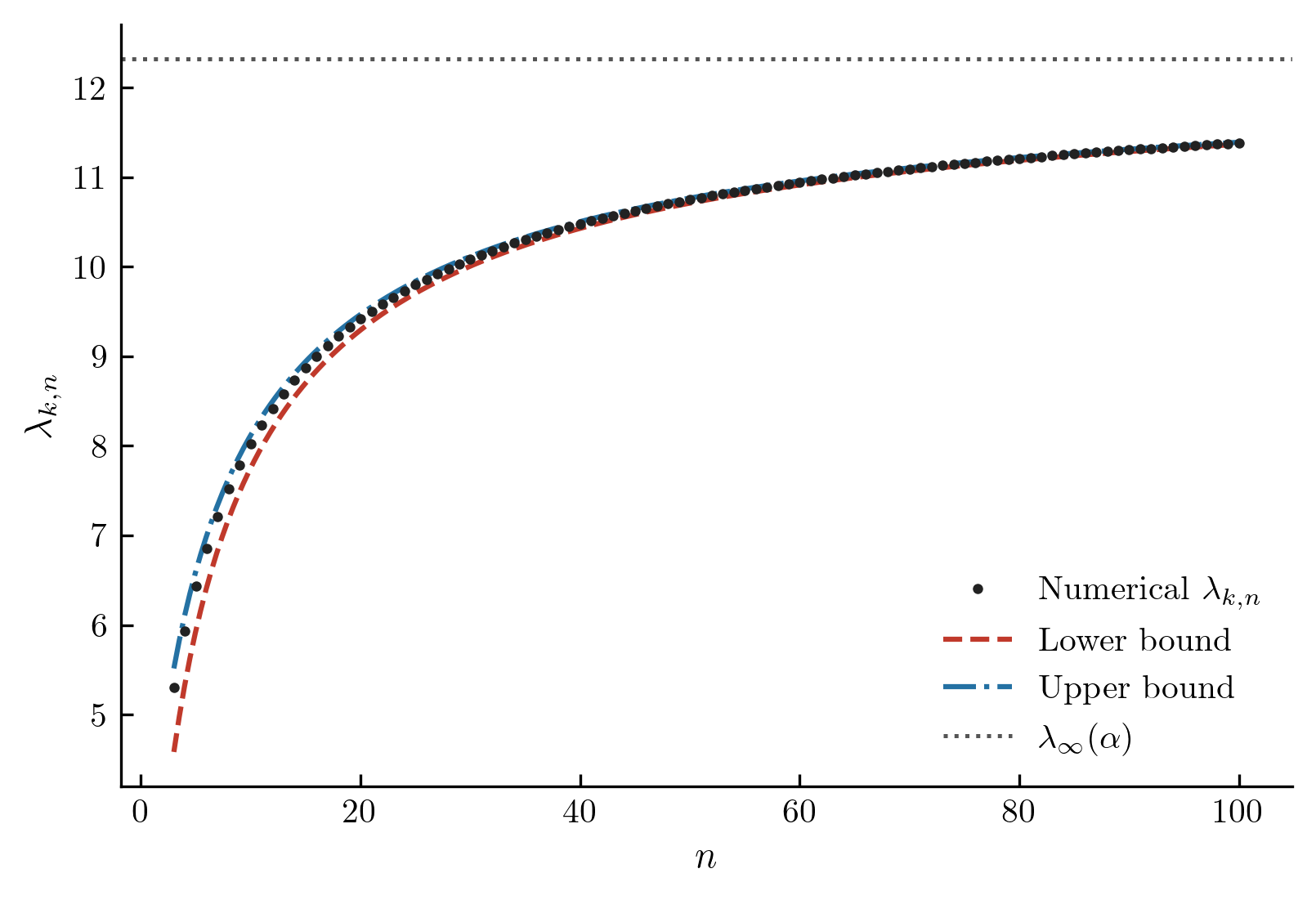}
   \end{minipage}\hfill
   \begin{minipage}{0.42\textwidth}
     \centering
     \includegraphics[width=\linewidth]{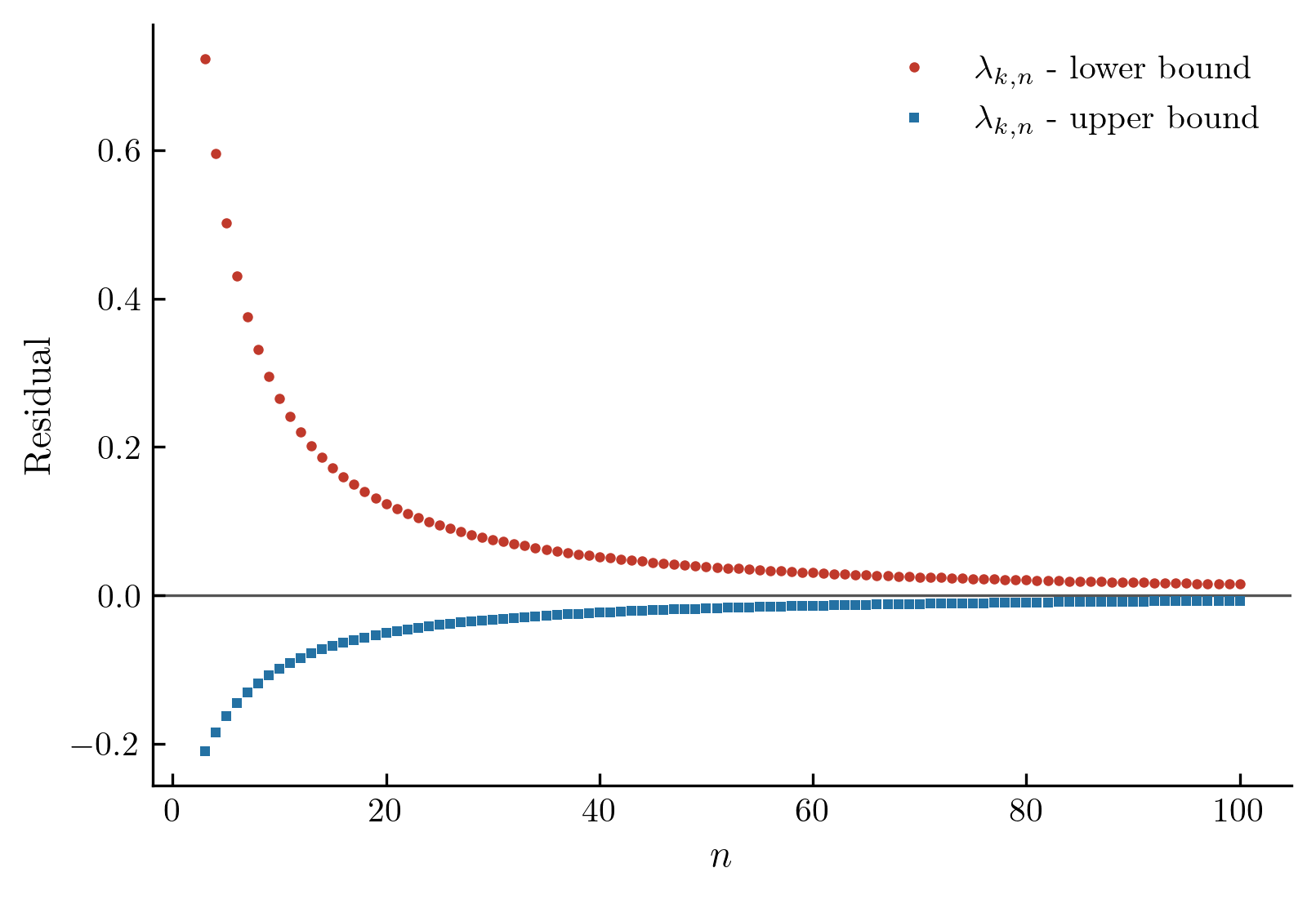}
   \end{minipage}
\caption{The left plot depicts selected numerically computed values of $\lambda_{k,n}$ when $\alpha = \frac{n}{k} = \frac{3}{2}$ is fixed and $1 \leq n \leq 100$; the bounds computed in Propositions \ref{lower bound} and \ref{Laplace prop}; and the limiting eigenvalue (Theorem \ref{limit theorem intro}). The right plot depicts the signed differences between the numerically computed value of $\lambda_{k,n}$ and the lower- and upper-bounds, for the same range.}
\end{figure}

\begin{figure}[H]
   \begin{minipage}{0.42\textwidth}
     \centering
     \includegraphics[width=\linewidth]{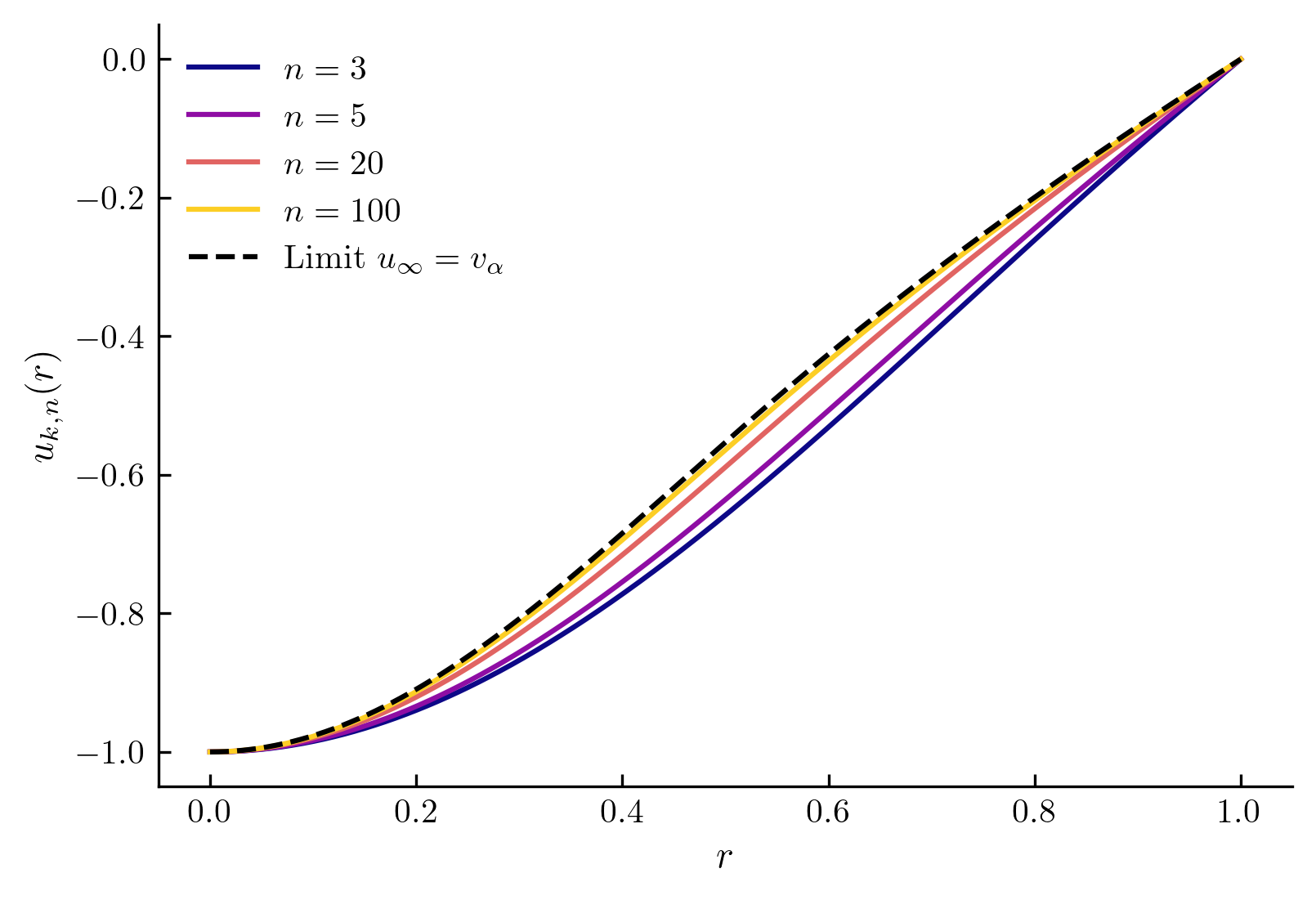}
   \end{minipage}\hfill
   \begin{minipage}{0.42\textwidth}
     \centering
     \includegraphics[width=\linewidth]{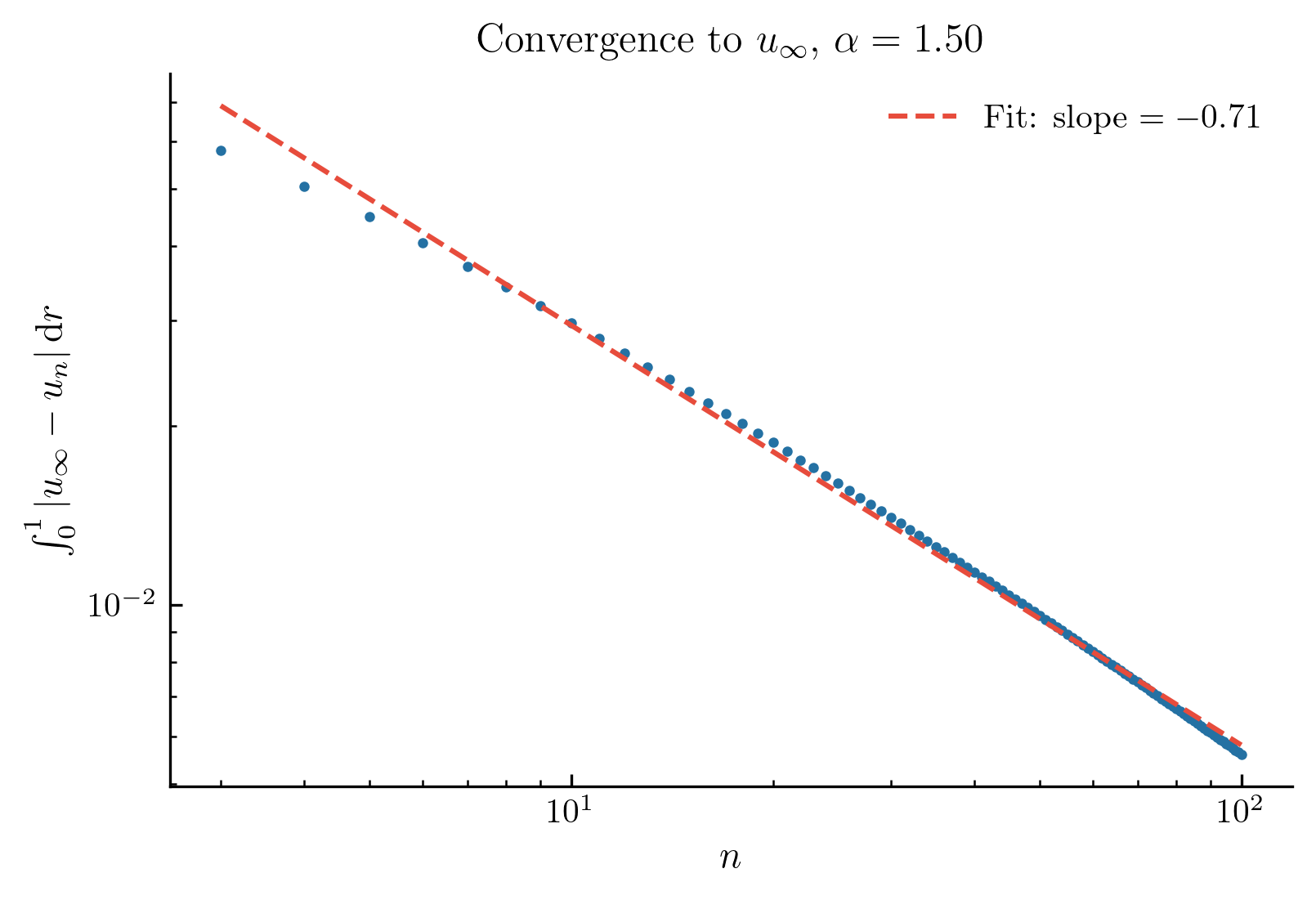}
   \end{minipage}
\caption{The left plot depicts selected numerically computed values of $u_{k,n}$ when $\alpha = \frac{n}{k} = \frac{3}{2}$ is fixed and $1 \leq n \leq 100$, and the limiting function $u_\infty := v_{3/2}$. The right plot depicts the $L^1$ distance between the numerically computed $u_{k,n}$ and $u_\infty$ for the same range, measured on a logarithmic scale, and a line of best fit.}
\end{figure}


\begin{figure}[H]
   \begin{minipage}{0.42\textwidth}
     \centering
     \includegraphics[width=\linewidth]{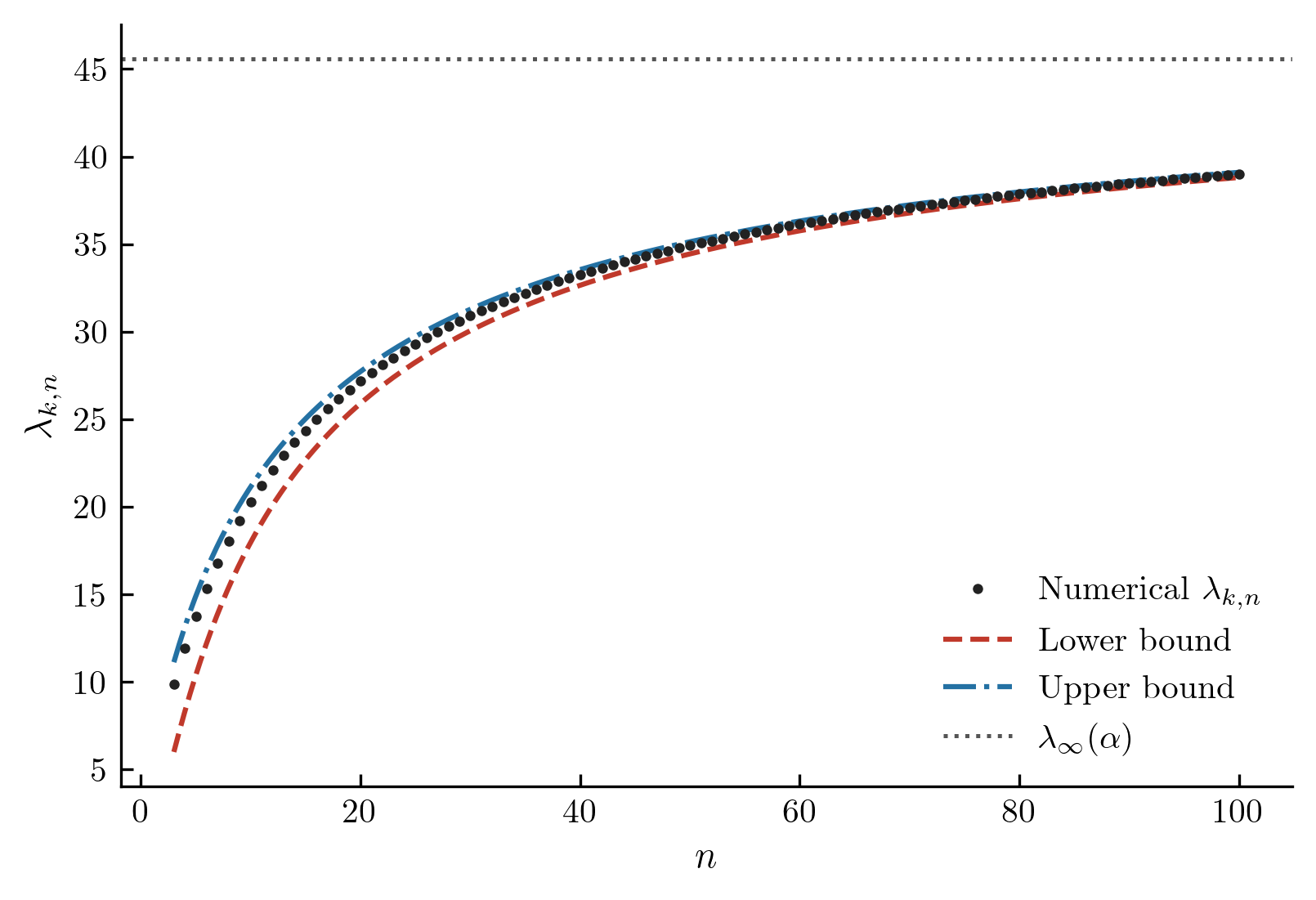}
   \end{minipage}\hfill
   \begin{minipage}{0.42\textwidth}
     \centering
     \includegraphics[width=\linewidth]{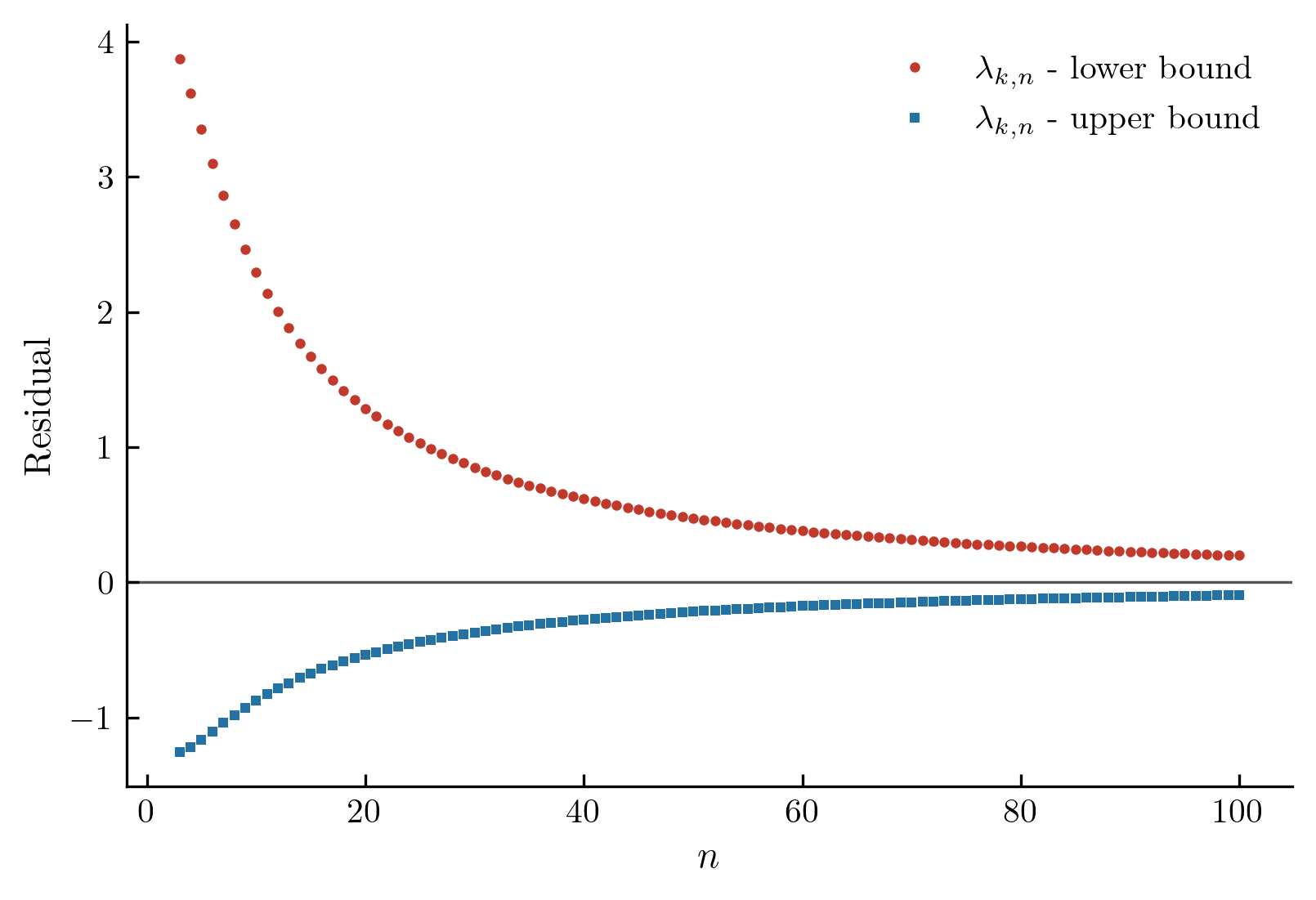}
   \end{minipage}
\caption{The left plot depicts selected numerically computed values of $\lambda_{k,n}$ when $\alpha = \frac{n}{k} = 3$ is fixed and $1 \leq n \leq 100$; the bounds computed in Propositions \ref{lower bound} and \ref{Laplace prop}; and the limiting eigenvalue (Theorem \ref{limit theorem intro}). The right plot depicts the signed differences between the numerically computed value of $\lambda_{k,n}$ and the lower- and upper-bounds, for the same range.}
\end{figure}

\begin{figure}[H]
   \begin{minipage}{0.42\textwidth}
     \centering
     \includegraphics[width=\linewidth]{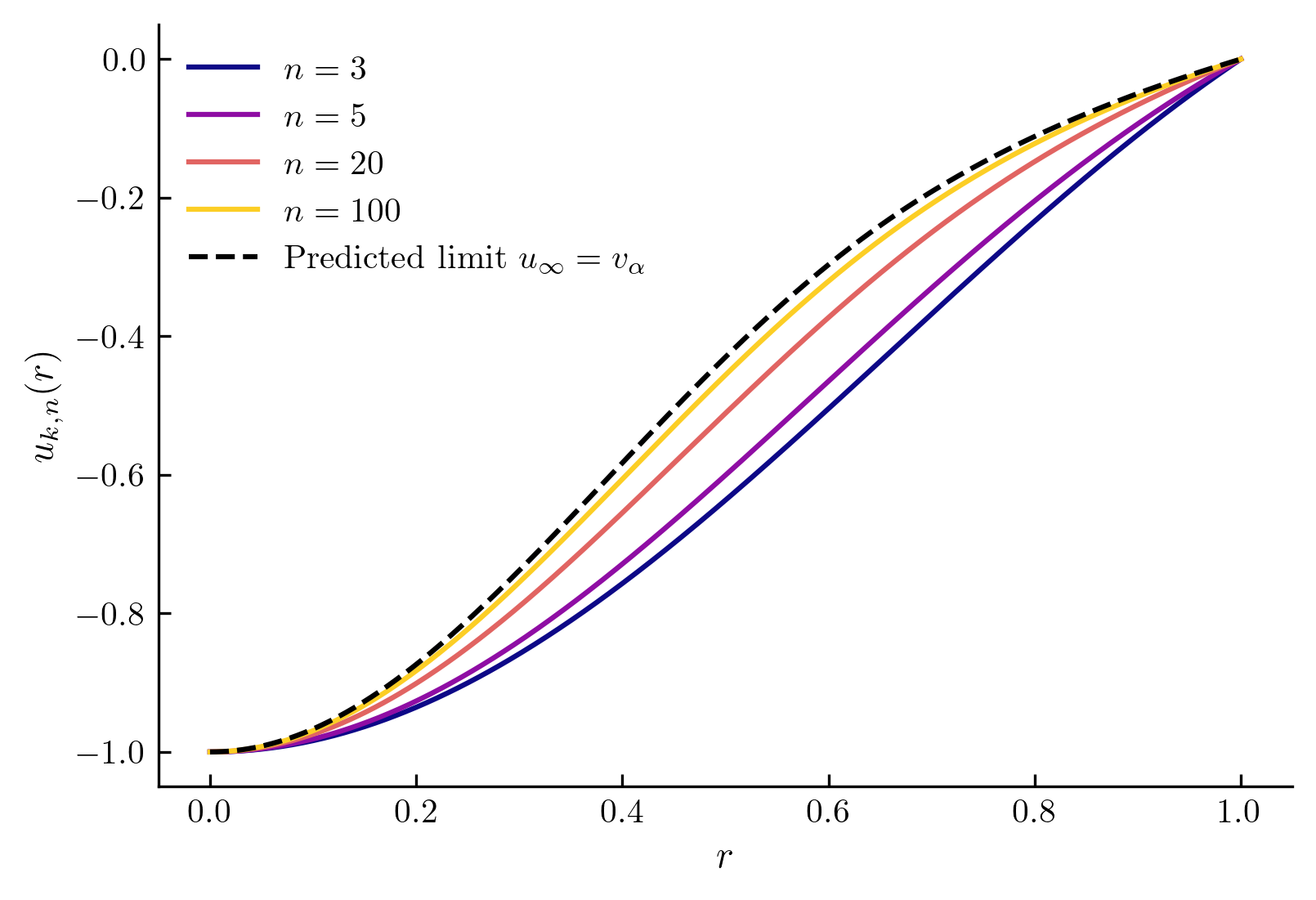}
   \end{minipage}\hfill
   \begin{minipage}{0.42\textwidth}
     \centering
     \includegraphics[width=\linewidth]{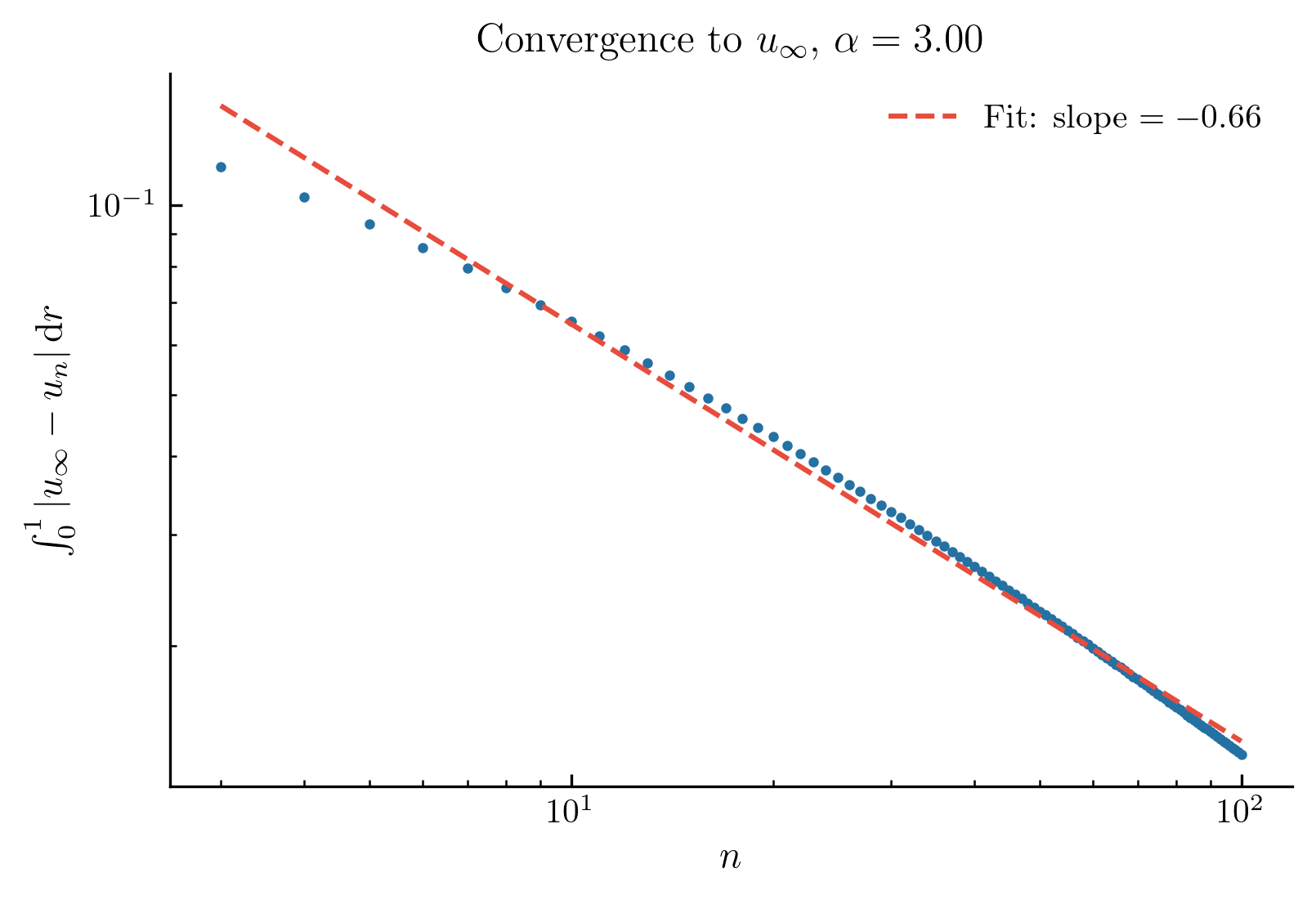}
   \end{minipage}
\caption{The left plot depicts selected numerically computed values of $u_{k,n}$ when $\alpha = \frac{n}{k} = 3$ is fixed and $1 \leq n \leq 100$, and the limiting function $u_\infty := v_{3}$. The right plot depicts the $L^1$ distance between the numerically computed $u_{k,n}$ and $u_\infty$ for the same range, measured on a logarithmic scale, and a line of best fit.}
\end{figure}


\noindent The numerical data suggests the following two statements should be true:
\begin{enumerate}
\item Theorem \ref{eigenfunction theorem intro} should hold for all $\alpha \geq 1$.
\item For each fixed $\alpha$, the eigenfunctions $u_{k, \alpha k}$ should increase to $v_\alpha$.
\end{enumerate}

Statement (2) has some relation to \cite[Question 5.1]{CF25}; there, it is observed that there is a number $R_n > 0$ such that the rescaled eigenfunction $\vp_n(r) := u_{n,n}(R^{-1}_n r)$, $0 \leq r \leq R_n$, satisfies
\[
\vp_{n}'(R_n) = R_n.
\]
Clearly $R_n^2 = u_{n,n}'(1)$; one might then ask if we can determine $u_{n,n}'(1)$ explicitly. Proving (2) would not do so explicitly, but it would show that $R_n$ is decreasing in $n$, with the sharp bounds:
\[
v_1'(1) = 2 e^{-\frac{1}{2}}\  \leq\  R_n^2\  \leq\  \frac{\pi}{2} = u_{1,1}'(1).
\]

Finally, it is additionally interesting to ask if one can determine the curve $k \mapsto \lambda_{k,\alpha k}$ implicitly, perhaps as the solution to another ODE.


\begin{thebibliography}{99}

\bibitem{BZ23a} Badiane, P., Zeriahi, A. {\em The Eigenvalue Problem for the complex Monge-Amp\`ere operator}, J. Geom. Anal. {\bf 33} (2023), no. 12, Paper No. 367, 44 pp.

\bibitem{BZ23b} Badiane, P., Zeriahi, A. {\em A variational approach to the eigenvalue problem for complex Hessian operators}, preprint, arXiv: 2306.04437.

\bibitem{BP21} Birindelli, I., Payne, K. R. {\em Principal eigenvalues for $k$-Hessian operators by maximum principle methods}. Math. Eng. {\bf 3} (2021), no. 3, Paper No. 021, 37 pp.

\bibitem{BNT09} Brandolini, B., Nitsch, C., Trombetti, C. {\em New isoperimetric estimates for solutions to Monge-Amp\`ere equations.} Ann. Inst. H. Poincar\'e Anal. Non Lin\'eaire {\bf 26} (2009), no. 4, 1265--1275.

\bibitem{CLYY25} Chen, L., Lin, Y., Yang J., Yi, W. {\em A Convergent Inexact Abedin-Kitagawa Iteration Method for Monge-Amp\`ere Eigenvalue Problems}. J. Sci. Comput. {\bf 106}, 39 (2026).

\bibitem{CLM26} Chu, J., Liu, Y., McCleerey, N. {\em The eigenvalue problem for the complex Hessian operator on $m$-pseudoconvex manifolds.} J. Funct. Anal. {\bf 290} (2026), no. 3, Paper No. 111258, 58 pp.

\bibitem{CF25} Collins, T., Firester, B. {\em On a general class of free boundary Monge-Amp\`ere equations}, preprint, arXiv: 2508.05551.

\bibitem{GLLQ20} Glowinski, R., Leung, S., Liu, H., Qian, J. {\em On the numerical solution of nonlinear eigenvalue problems for the Monge-Amp\`ere operator}. ESAIM Control Optim. Calc. Var. {\bf 26} (2020), Paper No. 118, 25 pp.

\bibitem{Le25} Le, Nam Q. {\em Large dimension behavior of the Hessian eigenvalues of the unit balls}, arXiv: 2505.09409, to appear in Kodai Math. J.

\bibitem{Le18} Le, N. Q. {\em The eigenvalue problem for the Monge-Amp\`ere operator on general bounded convex domains}, Ann. Sc. Norm. Super. Pisa Cl. Sci. (5) {\bf 18} (2018), no. 4, 1519--1559.

\bibitem{Le26} Le, N. Q. {\em A Variational Approach to Degenerate Monge-Amp\`ere Equations with Mixed Measures and Monotonicity}, preprint, 2026, arXiv: 2603.19114.

\bibitem{LS17} Le, N. Q., Savin, O. {\em Schauder estimates for degenerate Monge-Amp\`ere equations and smoothness of the eigenfunctions.} Invent. Math. {\bf 207} (2017), no. 1, 389--423.

\bibitem{Lio85} Lions, P.-L. {\em Two remarks on Monge-Amp\`ere equations}, Ann. Mat. Pura Appl. (4) {\bf 142} (1985), 263--275.

\bibitem{LLQ22} Liu, H., Leung, S., Qian, J. {\em An Efficient Operator-Splitting Method for the Eigenvalue Problem of the Monge-Amp\`ere Equation}. Commun. Optim. Theory. {\bf 2022} (2022), pp. 1--22.

\bibitem{LZ25} Lu, C. H., Zeriahi, A. {\em A new approach to the Monge-Amp\`ere eigenvalue problem}, preprint, 2025, arXiv: 2507.18409.

\bibitem{OLBC10} Olver, F. W. J., Lozier, D. W., Boisvert, R. F., Clark, C. W. (eds.) {\em NIST Handbook of Mathematical Functions}. Cambridge University Press, New York, 2010.

\bibitem{Tso89} Tso, K. {\em On symmetrization and Hessian equations.} J. Analyse Math. {\bf 52} (1989), 94--106.

\bibitem{Tso90} Tso, K. {\em On a real Monge-Amp\`ere functional}. Invent. Math. {\bf 101}, 425--448 (1990).

\bibitem{W94} Wang, X.-J. {\em A class of fully nonlinear elliptic equations and related functionals}, Indiana Univ. Math. J. {\bf 43} (1994), no. 1, 25--54.

\end{thebibliography}
\end{document}